\documentclass[oneside,english,12pt]{amsart}
\usepackage[T1]{fontenc}
\usepackage[latin9]{inputenc}
\usepackage{lmodern}
\usepackage[nomarginpar]{geometry}
\usepackage{prettyref}
\usepackage{textcomp}
\usepackage{amstext}
\usepackage{amsthm}
\usepackage{amssymb}
\usepackage{esint}
\usepackage{color}

\makeatletter
\numberwithin{equation}{section}
\numberwithin{figure}{section}
\theoremstyle{plain}
\newtheorem{thm}{\protect\theoremname}[section]
\theoremstyle{definition}
\newtheorem{defn}[thm]{\protect\definitionname}
\theoremstyle{plain}
\newtheorem{prop}[thm]{\protect\propositionname}
\theoremstyle{plain}
\newtheorem{cor}[thm]{\protect\corollaryname}
\theoremstyle{remark}
\newtheorem{rem}[thm]{\protect\remarkname}
\theoremstyle{plain}
\newtheorem{lem}[thm]{\protect\lemmaname}

\usepackage{hyperref}
\usepackage{nameref}
\usepackage{mathtools}
\mathtoolsset{showonlyrefs} 

\DeclareMathOperator{\divv }{div}

\DeclareMathOperator{\dist }{dist}

\makeatletter
\let\orgdescriptionlabel\descriptionlabel
\renewcommand*{\descriptionlabel}[1]{%
  \let\orglabel\label
  \let\label\@gobble
  \phantomsection
  \edef\@currentlabel{#1}%
  \let\label\orglabel
  \orgdescriptionlabel{#1}%
}
\makeatother

\makeatother

\usepackage{babel}
\addto\captionsenglish{\renewcommand{\corollaryname}{Corollary}}
\addto\captionsenglish{\renewcommand{\definitionname}{Definition}}
\addto\captionsenglish{\renewcommand{\lemmaname}{Lemma}}
\addto\captionsenglish{\renewcommand{\propositionname}{Proposition}}
\addto\captionsenglish{\renewcommand{\remarkname}{Remark}}
\addto\captionsenglish{\renewcommand{\theoremname}{Theorem}}
\providecommand{\corollaryname}{Corollary}
\providecommand{\definitionname}{Definition}
\providecommand{\lemmaname}{Lemma}
\providecommand{\propositionname}{Proposition}
\providecommand{\remarkname}{Remark}
\providecommand{\theoremname}{Theorem}

\begin{document}

\title[Optimal Regularity for Subelliptic No-Sign Obstacle-Type Problems]{Optimal Regularity of Solutions to No-Sign Obstacle-Type Problems
for the Sub-Laplacian}

\author{Valentino Magnani}
\address{Dipartimento di Matematica, Universit\`a di Pisa, Largo Bruno Pontecorvo 5, 56127, Pisa, Italy}
\email{valentino.magnani@unipi.it}

\author{Andreas Minne}
\address{KTH Royal Institute of Technology, 100 44 Stockholm, Sweden}
\email{minne@kth.se}

\begin{abstract}
We establish the optimal $C_{H}^{1,1}$ interior regularity of solutions
to
\[
\Delta_{H}u=f\chi_{\{u\ne0\}},
\]
where $\Delta_{H}$ denotes the sub-Laplacian operator in a stratified
group. We assume the weakest regularity condition on $f$, namely the group convolution $f*\Gamma$ is $C_{H}^{1,1}$, where $\Gamma$ is the fundamental
solution of $\Delta_{H}$. The $C_{H}^{1,1}$ regularity is understood
in the sense of Folland and Stein. In the classical Euclidean setting,
the first seeds of the above problem are already present in the 1991
paper of Sakai and are also related to quadrature domains. As a special
instance of our results, when $u$ is nonnegative and satisfies the
above equation we recover the $C_{H}^{1,1}$ regularity of solutions
to the obstacle problem in stratified groups, that was previously
established by Danielli, Garofalo and Salsa. Our regularity result
is sharp: it can be seen as the subelliptic counterpart of the $C^{1,1}$
regularity result due to Andersson, Lindgren and Shahgholian.
\end{abstract}

\thanks{V.M. acknowledges the support of the University of Pisa, Project PRA 2018 49. 
A.M. acknowledges the support of the Knut and Alice Wallenberg Foundation, Project KAW 2015.0380.}

\maketitle
\global\long\def\dist{\operatorname{dist}}
\foreignlanguage{english}{}\global\long\def\divv{\operatorname{div}}
\foreignlanguage{english}{}\global\long\def\iid{\operatorname{id}}
\foreignlanguage{english}{}\global\long\def\aad{\operatorname{ad}}
\foreignlanguage{english}{}\global\long\def\SL{\operatorname{SL}}
\foreignlanguage{english}{}\global\long\def\SO{\operatorname{SO}}
\foreignlanguage{english}{}\global\long\def\OO{\operatorname{O}}
\foreignlanguage{english}{}\global\long\def\Mob{\operatorname{Mob}}
\foreignlanguage{english}{}\global\long\def\n#1{\|#1\|}
\foreignlanguage{english}{}\global\long\def\ns#1{\|#1\|^{2}}
\foreignlanguage{english}{}\global\long\def\ip#1#2{\langle#1,#2\rangle}
\foreignlanguage{english}{}\global\long\def\Iim{\operatorname{Im}}
\foreignlanguage{english}{}\global\long\def\Rre{\operatorname{Re}}
\foreignlanguage{english}{}\global\long\def\Ric{\operatorname{Ric}}
\foreignlanguage{english}{}\global\long\def\trace{\operatorname{tr}}
\foreignlanguage{english}{}\global\long\def\diam{\operatorname{diam}}
\foreignlanguage{english}{}\global\long\def\vol{\operatorname{vol}}
\foreignlanguage{english}{}\global\long\def\exp{\operatorname{exp}}
\foreignlanguage{english}{}\global\long\def\Heis{H}
\foreignlanguage{english}{}\global\long\def\df{\vcentcolon=}
\foreignlanguage{english}{}\global\long\def\fd{=\vcentcolon}
\foreignlanguage{english}{}\global\long\def\dil#1{\delta_{#1}}
\foreignlanguage{english}{}\global\long\def\mult{}
\foreignlanguage{english}{}\global\long\def\G{\mathbb{G}}
\foreignlanguage{english}{}\global\long\def\R{\mathbb{R}}
\foreignlanguage{english}{}\global\long\def\N{\mathbb{N}}
\newcommand{\der}{\partial}
\foreignlanguage{english}{}\global\long\def\ph{\varphi}
\foreignlanguage{english}{}\global\long\def\sm{\operatorname{\setminus}}
\foreignlanguage{english}{}\global\long\def\B{\mathbb{\mathbb{B}}}

\selectlanguage{english}%

\section{Introduction}

The main question we consider in this paper is the optimal interior
regularity of distributional solutions to the \emph{no-sign obstacle-type problem}
\begin{equation}
\Delta_{H}u=f\chi_{\{u\ne0\}}\label{eq:no-sign-obstacle-problem-Carnot}
\end{equation}
on some domain of a stratified group $\G$, see \prettyref{sec:Setup-and-Notation} for notation and terminology. In the Euclidean setting, the obstacle
problem is among the most studied topics in the field of Free Boundary
Problems, see for instance the monographs by Rodrigues \cite{MR880369},
Friedman \cite{MR1009785}, and Petrosyan et al. \cite{MR2962060}.
It asks which properties can be deduced about a function with given
boundary values and that minimizes the Dirichlet energy, under the
constraint of lying above a given function. This is the classical
obstacle problem, that can be studied through the theory of variational
inequalities, using the Dirichlet energy, see for instance \cite{MR1786735,MR0318650}.
The variational approach, after subtracting the obstacle from the
solution, leads to the following PDE formulation of the problem
\begin{equation}
\begin{cases}
\Delta u=f\chi_{\{u>0\}} & \text{in }B_{1},\\
\phantom{\Delta}u\ge0 & \text{in \ensuremath{B_{1},}}\\
\phantom{\Delta}u=g & \text{on }\partial B_{1},
\end{cases}\label{eq:classical_obstacle}
\end{equation}
where $B_1$ denotes the metric unit ball with respect to the
Carnot--Carath\'eodory distance (Definition~\ref{def:metric ball}).
Our problem is a non-variational counterpart of \eqref{eq:classical_obstacle}, that is
\begin{equation}
\begin{cases}
\Delta u=f\chi_{\{u\ne0\}} & \text{in }B_{1},\\
\phantom{\Delta}u=g & \text{on }\partial B_{1}.
\end{cases}\label{eq:no-sign-obstacle-problem}
\end{equation}
We point out that \eqref{eq:no-sign-obstacle-problem} \textemdash{}
which is called a no-sign obstacle-type problem \textemdash{} naturally
appears also when considering the so-called quadrature domains \cite{MR1097025,MR2129734}.

Two important questions on this problem concern the regularity of
solutions to \eqref{eq:no-sign-obstacle-problem}, and the regularity
of the free boundary. In Euclidean space, the analysis of both questions
is essentially complete \cite{MR1097025,MR1745013,MR2369892,MR2999297}.
In particular, in relation to the regularity of solutions, Andersson
et al. \cite{MR2999297} show that $u$ has the optimal $C^{1,1}$
regularity if the linear problem $\Delta v=f$ has a $C^{1,1}$ solution.
This is the minimal regularity assumption on $f$ in order to establish
the $C^{1,1}$ regularity of solutions.

The main result of this paper is the sharp regularity of solutions
to \eqref{eq:no-sign-obstacle-problem-Carnot} also in the subelliptic
setting of stratified groups.
\begin{thm}
[$C_H^{1,1}$ regularity]\label{thm:main-theorem} Let $u\in  L^{\infty}(B_{1})$
be a distributional solution to \eqref{eq:no-sign-obstacle-problem-Carnot} in the unit ball $B_1$. 
Let $f:B_{1}\to\R$ be locally summable such that $f*\Gamma\in C_{H}^{1,1}(B_{1})$.
Then there exists a universal constant $C>0$ such that, after a modification on a negligible set, we have $u\in C_{H}^{1,1}(B_{1/4})$ and
\begin{equation}
\lVert D_{h}^{2}u\rVert_{L^{\infty}(B_{1/4})}\le C\left(\lVert D_{h}^{2}(f*\Gamma)\rVert_{L^{\infty}(B_{1})}+\lVert u\rVert_{L^{\infty}(B_{1})}\right).\label{eq:C11-estimate}
\end{equation}
\end{thm}
In our setting, the natural counterpart of the Euclidean $C^{1,1}$ regularity is the
$C^{1,1}_H$ regularity, where the horizontal derivatives are required to be Lipschitz 
continuous (Definition~\ref{def:C1alpha}).
The function $\Gamma$ denotes the fundamental solution of $\Delta_H$ (Definition~\ref{def:FundSol}). For further notation and terminology, we address the reader to \prettyref{sec:Setup-and-Notation}.

We wish to emphasize that $u$ satisfies \eqref{eq:no-sign-obstacle-problem-Carnot} also in the strong sense. Indeed, the distributional equality $\Delta_H(f*\Gamma)=-f$ joined with the assumed $C^{1,1}_H$ regularity of $f*\Gamma$ show that $f\in L^\infty(B_1)$.
Therefore $f\chi_{\{u\ne0\}}\in L^\infty(B_1)$ and by the regularity result of Folland, \cite[Theorem~6.1]{MR0494315}, we get $u\in W^{2,p}_{H,\text{loc}}(B_1)$ 
for every $1\le p<\infty$.
The $C_{H}^{1,1}$ regularity of solutions to the obstacle problem
in stratified groups was obtained by Danielli et al. \cite{MR1976081},
using the variational formulation of the problem. The regularity of
the free boundary was subsequently established in step two groups
\cite{MR2323535}. Further results in this area have been obtained
for Kolmogorov operators and parabolic non-divergence form operators
of H\"ormander type \cite{MR2434049,MR2658144,MR2891355,MR3093377}.
The no-sign obstacle-type problem in terms of the equation \eqref{eq:no-sign-obstacle-problem-Carnot}
does not seem to have been considered before in the subelliptic setting. 

Our arguments are remarkably different from the ones used for the
obstacle problem. For instance, in this problem without a forcing term
the solution is automatically superharmonic with respect to $\Delta_{H}$,
while in our setting we have no such sign condition that would yield 
a superharmonic solution. We initiate our
analysis observing that second order horizontal derivatives of solutions
to \eqref{eq:no-sign-obstacle-problem-Carnot} satisfy certain $BMO$
estimates, that have been established by Bramanti et al. \cite{MR2174915,MR3100057}.
The subsequent step is to construct suitable approximating polynomials,
starting from the second order horizontal derivatives of the solution.
Indeed these polynomials yield a subquadratic growth estimate \eqref{eq:quadratic-growth-1}
at small scales. We point out that this estimate is valid for any
bounded and $W^{2,p}$ regular function, with bounded sub-Laplacian, so it might be of
independent interest. As a consequence, we perform a suitable rescaling
of the equation and then infer the crucial decay estimate of the measure of the coincidence
set (Proposition~\ref{prop:decay-of-coincidence-set}), when the
horizontal Hessian of the approximating polynomial is sufficiently
large. More details on this procedure can be found at the beginning
of Section~\ref{sec:C11}.

Although our ideas mainly follow the path set up by Andersson et al.
\cite{MR2999297}, and Figalli and Shahgholian \cite{figsha}, there
are several difficulties related to the subelliptic setting. The basic
one is concerned with the fact that the sub-Laplacian $\Delta_{H}$
is degenerate elliptic. In addition, since the operator $\Delta_{H}$
is written in terms of H\"omander vector fields, we can only consider
the \emph{horizontal Hessian} \eqref{eq:grad_hess} of the solution,
that is a nonsymmetric matrix. Then the construction of the approximating
polynomials starting from the average of the second order noncommuting
derivatives $X_{i}X_{j}u$ becomes more delicate and requires some
preliminary algebraic work, see Section \ref{sec:Background-Preliminaries}.
Notwithstanding the technical complications, the proof has become
more streamlined: we can stay clear of the projection operator used
in \cite{MR2999297}, and this simplifies several technical points.
A suitable quantitative decay estimate of the zero-level set \eqref{eq:Decayr}
can be obtained also in our setting. Finally, we adapt Caffarelli's polynomial 
iteration technique of \cite{MR1005611} to
find explicit estimates of the second order horizontal derivatives,
see \eqref{eq:C11-estimate}.

We finish the introduction by giving an overview of the paper. In Section 2 we introduce some basic notions on stratified groups and the related function spaces.
In Section~3 we construct suitable second order homogeneous polynomials (Definition~\ref{def:c_ijrx_0}), that have an assigned horizontal Hessian (Corollary~\ref{c:2ndorderpolyn}). Then some important $W^{2,p}$ and
$BMO$ estimates are presented. Finally, we provide the crucial scaling estimates of Lemma~\ref{lem:dyadic-scale}. In Section~4 we prove a subquadratic growth estimate of the difference between a solution and its approximating polynomial. Then we apply the subquadratic growth estimate to get a suitable decay of the measure
of the zero-level set. Finally, we establish the $C_{H}^{1,1}$ regularity
in quantitative terms, according to \eqref{eq:C11-estimate}. 

\vskip3mm\noindent
{\bf Acknowledgements.} 
A substantial part of this work was accomplished during 
the time A.M. spent at Scuola Normale Superiore in Pisa, with
the support of the Knut and Alice Wallenberg Foundation.
V.M. gratefully thanks Henrik Shahgholian, along
with KTH Royal Institute of Technology, for support and hospitality 
during December 2018, where important parts of the present work were 
carried out. He also acknowledges the Institutional Research Grant from
the University of Pisa. 
The authors wish to thank Marco Bramanti
for fruitful discussions concerning some a priori $BMO$ estimates.
They are also indebted to the referee for several useful remarks.

\section{Basic facts and notation \label{sec:Setup-and-Notation}}

A \emph{stratified group} is a simply connected, real nilpotent Lie
group $\mathbb{G}$, whose Lie algebra $\mathcal{G}$ has a special
stratification. We denote by $V_{i}$ the subspaces of $\mathcal{G}$,
having the properties:
\[
\mathcal{G}=V_{1}\oplus V_{2}\oplus\cdots\oplus V_{\iota}\quad\text{and}\quad[V_{1},V_{j}]=V_{j+1}
\]
for $j=1,\ldots,\iota$ and $V_{\iota+1}={0}.$ Let us denote by $n$ the topological dimension of $\G$ and by $m$ the dimension of $V_1$. We choose a \emph{graded basis
$X_{1},X_{2},\ldots,X_{n}$ }of $\mathcal{G}$, that is characterized
by the property that
\[
X_{m_{j-1}+1},\ldots,X_{m_{j}}
\]
is a basis of $V_{j}$ for all $j=1,\ldots,\iota$, where we have
set $m_{0}=0$, $m_{1}=m$ and $m_{j}=\sum_{i=1}^{j}\dim V_{i}$.
We notice that $m_{\iota}=n$ and with these definitions, if $m_{k-1}<j\le m_{k}$,
then $k\in\N$ is uniquely determined and we define the positive integer
\begin{equation}\label{eq:degree}
d_{j}\df k.
\end{equation}
Through the exponential mapping of $\G$, one can construct a diffeomorphism
from $\R^{n}$ to $\G$. Hence we have defined a \emph{graded basis
}$e_{1},e_{2},\ldots,e_{n}$ of $\R^{n}$ and \emph{graded coordinates}
$x_{1},x_{2},\ldots,x_{n}$ that define the point $x=(x_{1},x_{2},\ldots,x_{n})$
of $\G$. This allows us to identify $\G$ with $\R^{n}$, as it will
be understood in the sequel. 

In addition, one may also verify that the Lebesgue measure of $\R^n$ through the graded coordinates yields the Haar measure of the group $\G$.  The notation $|A|$ denotes the Lebesgue measure of a measurable set $A\subset\R^n$.

The diffeomorphism associated to graded coordinates has also the property 
that the group operation on $\G$, when read in $\R^{n}$,
is given by a special polynomial group operation 
\begin{equation}
xy=x+y+BCH(x,y),\label{eq:BCH}
\end{equation}
where the precise form of the vector polynomial $BCH:\R^{n}\times\R^{n}\to\R^{n}$
is given by the important \emph{Baker\textendash Campbell\textendash Hausdorff
formula}, in short BCH formula, see for instance \cite{MR746308}. The $degree$ of $x_{j}$
is the integer $d_{j}$ defined in \eqref{eq:degree} and we define
\emph{intrinsic dilations} as follows
\[
\delta_{r}x=(rx_{1},\ldots,rx_{m},r^{2}x_{m+1},\ldots,r^{2}x_{m_{2}},\ldots,r^{\iota}x_{m_{\iota-1}+1},\ldots,r^{\iota}x_{n})=\sum_{j=1}^{n}r^{d_{j}}x_{j}e_{j}
\]
for any $r>0$. The notion of degree fits the algebraic properties of dilations, since  
\begin{equation}
\delta_{r}(xy)=(\delta_{r}x)(\delta_{r}y)\label{eq:delta_r}
\end{equation}
for all $x,y\in\R^{n}$. 
By the form of dilations, for every measurable set $A\subset \R^n$ we have
\[
|\delta_r(A)|=r^Q|A|
\]
for all $r>0$, where $Q=\sum_{j=1}^n d_j$ can be proved to be the Hausdorff dimension of $\G$.

The metric structure of $\R^{n}$ is given by a control distance.
We say that $\gamma:[0,T]\to\R^{n}$, an absolutely continuous curve,
is \textit{admissible }if for a.e. $t\in[0,1]$ there holds 
\[
\dot{\gamma}(t)=\sum_{i=1}^{m}b_{i}(t)X_{i}(\gamma(t))
\]
and $\sum_{i=1}^{m}b_{i}(t)^{2}\le1.$ We denote by $\mathcal{H}(x,y)$
the family of all admissible curves whose image contains ${x,y}$.
By Chow's theorem, $\mathcal{H}(x,y)$ is nonempty for every
$x,y\in\R^{n}$, hence the following ``control distance'' 
\[
d(x,y)=\inf\left\{ T>0\Big|\gamma:[0,T]\to\R^{n},\ \gamma\in\mathcal{H}(x,y)\right\}
\]
is well defined. It is also possible to check that $d$ is actually a distance, corresponding to the well known 
\textit{Carnot\textendash Carathéodory distance}.

Since left translations preserve the ``horizontal velocity'', $d$ is also {\em left invariant}, namely $d(x,y)=d(zx,zy)$ for all $x,y,z\in\R^{n}$.
Furthermore, dilations are Lie group homomorphisms, hence the Carnot--Carathéodory distance is homogeneous in the sense that $d(\delta_{r}x,\delta_{r}y)=rd(x,y)$
for every $x,y\in\R^{n}$ and $r>0.$ 

\begin{defn}[Metric balls]\label{def:metric ball}
For $x\in\R^n$ and $r>0$, we denote by $B_{r}(x)$ the open ball of center $x$ and radius $r>0$ with respect to $d$. Precisely, this is the set $\{y\in\R^n:d(x,y)<r\}$. When $x=0$, we use the notation $B_{r}\df B_{r}(0)$. 
\end{defn}
From the properties of $d$ and $\delta_{r}$, 
it is easy to observe that 
\[
B_{r}(x)= x\dil r(B_{1}).
\]
Dilations also allow us to introduce a natural notion
of homogeneity, so we may say that a polynomial $p:\R^{n}\to\mathbb{R}$
is \emph{$k$-homogeneous} if
\[
p(\delta_{r}x)=r^{k}p(x)\qquad\text{for all \ensuremath{x\in\R^{n}} and \ensuremath{r>0}.}
\]
The number $k\in\N$ is the \emph{degree} of $p$. 
Moreover, any vector field $X_{j}$ of the fixed graded basis can be identified
with a first order differential operator of the form 
\begin{equation}
X_{j}=\partial_{x_{j}}+\sum_{i=m_{d_{j}}+1}^{n}a_{ji}\partial_{x_{i}}\label{eq:vectfields}
\end{equation}
for every $j=1,\ldots,n$. The functions $a_{ji}:\R^{n}\to\R$ are
homogeneous polynomials of degree $d_{i}-d_{j}\ge1$ and in particular
$X_{j}(0)=e_{j}$ for all $j=1,\ldots,n$.
In the sequel $\Omega$ will be understood as an open bounded subset of $\G$, that can be also identified with an open subset 
of $\R^{n}$, if not otherwise stated. 

Given a function $u:\Omega\to\R$ and considering the vector fields
$X_{j}$ as differential operators, we may introduce the \textit{horizontal
gradient} and the \textit{horizontal Hessian}
\begin{equation}
\nabla_{h}u=(X_{1}u,\ldots,X_{m}u)\quad\text{and}\quad D_{h}^{2}u=\left(\begin{array}{cccc}
X_{1}X_{1}u & X_{1}X_{2}u & \cdots & X_{1}X_{m}u\\
X_{2}X_{1}u & X_{2}X_{2}u & \cdots & X_{2}X_{m}u\\
\vdots & \vdots & \ddots & \vdots\\
X_{m}X_{1}u & \cdots & \cdots & X_{m}X_{m}u
\end{array}\right),\label{eq:grad_hess}
\end{equation}
respectively, whenever they are pointwise defined.
More generally, we can define higher order differential operators considering
for $I=(i_1,\ldots,i_n)\in\N^n$ the function
\[
X^Iu\df X_1^{i_n}\cdots X_n^{i_1}u.
\]
\begin{defn}[Folland--Stein spaces]
\label{def:C1alpha}
Let $\Omega\subset \R^n$ be an open set. 
We denote by $C^1_H(\Omega)$ the space of all functions $u:\Omega\to\R$ such that the horizontal derivatives $X_ju$ exist on $\Omega$ for all $j=1,\ldots,m$ and are continuous.
If $0<\alpha\le 1$, then $C^{1,\alpha}_H(\Omega)$ is the space of functions $u$ in $C^1_H(\Omega)$
such that there exists $C>0$ with the property that 
\[
|X_jf(x)-X_jf(y)|\le C\,d(x,y)^\alpha
\]
for every $x,y\in\Omega$ and $j=1,\ldots,m$.
\end{defn}
Notice that $D_{h}^{2}u$ is not symmetric, since the vector fields $X_{j}$ do not commute in general. We say that $X_{j}u$ are the \emph{horizontal
derivatives }and $X_{i}X_{j}u$ are the \emph{second order horizontal
derivatives.} The {\em symmetrized horizontal Hessian} is defined as 
\[
D_{h}^{2,s}u=\frac{1}{2}\left(D_{h}^{2}u+D_{h}^{2}u^{T}\right).
\]
The sub-Laplacian is defined as
\[
\Delta_{H}u=\sum_{j=1}^{m}X_{j}^{2}u.
\]
Functions satisfying $\Delta_{H}u=0$ are called as usual \emph{harmonic
functions}. 
\begin{defn}[Fundamental solution]\label{def:FundSol}
We say that $\Gamma\in C^\infty(\G\sm\{0\}) $ is a 
{\em fundamental solution} for $\Delta_H$ if 
it is locally summable, it vanishes at infinity and satisfies
$\Delta_H\Gamma=-\delta_0$, where $\delta_0$ denotes
the Dirac distribution centered at the origin.
\end{defn}

The fundamental solution $\Gamma$ defines a gauge $d_{G}=\Gamma^{1/(2-Q)}$, that is 1-homogeneous with respect to dilations and continuous on $\R^n$.
We can readily check that there exists a constant $c_0>1$ such that 
\begin{equation}\label{eq:C_0}
c_0^{-1}d_G(x)\le d(x,0)\le c_0 d_G(x)
\end{equation}
for all $x\in\R^n$. Defining $d_G(x,y)\df d_G(x^{-1}y)$  we also introduce the {\em gauge ball}
\begin{equation}\label{eq:BGrx}
B_{r}^{G}(x)=\left\{ y\in\R^{n}:d_{G}(x,y)<r\right\}.
\end{equation}
The previous estimates clearly imply that 
\begin{equation}
B_{r}^{G}(x)\subset B_{c_{0}r}(x)\label{eq:defc_0}
\end{equation}
for every $r>0$ and $x\in\R^{n}$.

\begin{prop}\label{prop:derivEst}
Let $\Omega\subset\R^n$ be an open set and let $\vartheta$ be harmonic in $\Omega$.
We consider an open set $\Omega'\subset\Omega$ and $h>0$ such that
\[
\dist_G(\Omega',\Omega^c)\df\inf\left\{d_G(x,y):x\in\Omega',\,y\in\Omega^c\right\}> h. 
\]
Then $\vartheta\in C^\infty(\Omega)$ and for every multiindex $I$ there exists a constant $C_{I,h}>0$ such that
\begin{equation}\label{eq:Xalpha}
|X^I\vartheta(x)|\le C_{I,h} \|\vartheta\|_{L^1(\Omega)}
\end{equation}
\end{prop}
\begin{proof}
We consider the function $\phi$ defined in \cite[(5.50e)]{MR2363343}, where we choose 
$\ph$ appearing in the definition of $\phi$, such that $\ph\in C_c^\infty(]3^{-1},1[)$,
$\ph\ge0$ and $\int_\R\ph(t)dt=1$. It follows that $\phi$ is smooth and bounded on $\R^n$, 
along with all of its derivatives, and it is compactly supported in $B^G_1$, 
see the definition \eqref{eq:BGrx}. We also consider $\phi_r(z)\df r^{-Q}\phi(\delta_{1/r}z)$,
that is compactly supported on $B^G_r$. We finally set $\hat\phi_r(z)\df\phi_r(z^{-1})$ for all $z\in\R^n$ and $r>0$.
Thus, using \cite[(5.50a),(5.50d)]{MR2363343}, for every $x\in B_\lambda$, we get
\[
\vartheta(x)=\int_{B^G_{h}(x)} \phi_{h}(x^{-1}y)\vartheta(y)dy=
\int_\Omega \hat\phi_{h}(y^{-1}x)\vartheta(y)dy.
\]
We can differentiate the last integral an arbitrary number of times, due to the smoothness of $\hat\phi_h$, getting the smoothness of $\vartheta$ and the following estimate
\begin{equation}
|X^I\vartheta(x)|=\left|\int_{\Omega} X^I\hat\phi_{h}(y^{-1}x)\vartheta(y)dy\right|
\le 
\|X^I\hat\phi_h\|_{L^\infty(\R^n)}\|\vartheta\|_{L^1(\Omega)}
\end{equation}
for all $x\in\Omega'$. This concludes the proof.
\end{proof}

In our setting, we need the notion of Sobolev function adapted to
the horizontal vector fields $X_{1},\ldots,X_{m}$, see \cite{MR0494315}.
The horizontal Sobolev space $W_{H}^{k,p}(\Omega)$ consists of those functions
$u\in L^{p}(\Omega)$ for which, for all $j_{s}\in\{1,\ldots,m\}$
and $s\in\{1,\ldots,k\}$, there exists a function $v_{j_{1},\ldots,j_{k}}\in L^{p}(\Omega)$
such that
\[
\int_{\Omega}u(y)(X_{j_{1}}\cdots X_{j_{k}}\phi)(y)\,dy=(-1)^{k}\int_{\Omega}v_{j_{1},\ldots,j_{k}}(y)\phi(y)\,dy
\]
for any function $\phi\in C_{c}^{\infty}(\Omega)$.
Also in the more general setting of H\"ormander vector fields some 
Sobolev embedding theorems hold, see \cite[Theorem 1.11 and (3.19)]{MR1404326}, or
\cite[Theorem 1.1]{MR1425620}. The next theorem specializes these
embedding results for stratified groups. 
\begin{thm}\label{t:embed}
Let $p>Q$, where $Q$ is the Hausdorff dimension of $\G$
and let $\Omega'\Subset\Omega$ be any open and relatively compact subset.
Then there exists $C>0$, depending on $\Omega'$, such that for every $u\in W_{H}^{1,p}(\Omega)$,
up to a modification of $u$ on a negligible set, we have
\[
|u(x)-u(y)|\le C\,\|u\|_{W_{H}^{1,p}(\Omega)}\,d(x,y)^{1-\frac{Q}{p}},
\]
for every $x,y\in\Omega'$.
\end{thm}
The following (1,1)-Poincar\'e inequality holds,
\begin{equation}
\int_{B_{r}(x)}|u(y)-u_{B_{r}(x)}|dy\le cr\int_{B_{r}(x)}|\nabla_{h}u(y)|dy\label{eq:Poincare}
\end{equation}
for every $u\in C^{1}(\overline{B_{r}(x)}).$ This inequality follows
from \cite{MR850547}, see also \cite{MR1785405}.

For any measurable function
$u$ that is summable on a measurable set $A\subset\Omega$, we use
the notation 
\[
u_{A}\df\fint_{A}u(y)\,dy=\frac{1}{\lvert A\rvert}\int_{A}u(y)\,dy.
\]
\begin{defn}
\label{def:BMOnorms} For $u\in L^{1}(\Omega)$, we define
the $BMO$ seminorms 
\begin{align*}
[u]_{BMO(\Omega)} & \df\sup_{x_{0}\in\Omega,r>0}\fint_{B_{r}(x_{0})\cap\Omega}|u(y)-u_{B_{r}(x_{0})}|\,dy,\\{}
[u]_{BMO_{\text{loc}}(\Omega)} & \df\sup_{B_{r}(x_{0})\subset\Omega}\fint_{B_{r}(x_{0})}|u(y)-u_{B_{r}(x_{0})}|\,dy
\end{align*}
and for $1\le p<\infty$ the corresponding $BMO^p$ norms 
\begin{align*}
\lVert u\rVert_{BMO^{p}(\Omega)} & \df[u]_{BMO(\Omega)}+\lVert u\rVert_{L^{p}(\Omega)},\\
\lVert u\rVert_{BMO_{\text{loc}}^{p}(\Omega)} & \df[u]_{BMO_{\text{loc}}(\Omega)}+\lVert u\rVert_{L^{p}(\Omega)}.
\end{align*}
The spaces $BMO^p(\Omega)$ and $BMO_{\text{loc}}^p(\Omega)$ consist of all $L^p$ functions on $\Omega$ with finite $BMO^p$ and $BMO_{\text{loc}}^p$ norm, respectively. 
See \cite{MR3100057} for more information on $BMO$ functions in the
subelliptic setting. 
\end{defn}

\section{Preparatory results \label{sec:Background-Preliminaries}}

We first study the relationship between the coefficients
of a 2-homogeneous polynomial and its second order horizontal derivatives.
Then, by some $W_{H}^{2,p}$ and $BMO$ estimates, we show how to 
control the horizontal Hessian of a Sobolev
function by the horizontal Hessian of a suitable 2-homogeneous 
harmonic polynomial (Corollary~\ref{cor:BMOest}). Finally,
in Lemma~\ref{lem:dyadic-scale} we establish a quantitative control 
on the growth of these polynomials at small scales.

We need first to find 2-homogeneous polynomials with assigned second
order horizontal derivatives. To do this, we first observe that 
\eqref{eq:BCH}, combined with \eqref{eq:delta_r} and \eqref{eq:degree}, setting 
\[
BCH(x,y)=\sum_{j=m+1}^{n}q_{j}(x,y)e_{j},
\]
imply that any $q_{j}$ is a homogeneous polynomial of degree $d_{j}$.
Due to the BCH formula, one can also prove that $q_{l}$ is a 2-homogeneous polynomial
with respect to the variables $x_{1},\ldots,x_{m},y_{1},\ldots,y_{m}$
for all $l=m+1,\ldots,m_{2}$ and 
\[
q_{l}(x,y)=-q_{l}(y,x).
\]
From the definition of left invariant vector field, we get 
\[
a_{jl}(x)=\frac{\der q_{l}}{\der y_{j}}(x,0),
\]
for all $j=1,\ldots,m$ and $l=m+1,\ldots,m_{2}$. As a consequence,
we get 
\begin{equation}
\frac{\der a_{jl}}{\der x_{i}}=\frac{\der^{2}q_{l}}{\der x_{i}\der y_{j}}=-\frac{\der^{2}q_{l}}{\der x_{j}\der y_{i}}=-\frac{\der a_{il}}{\der x_{j}}\label{eq:antisymmcoeff}
\end{equation}
for all $i,j=1,\ldots,m$ and $l=m+1,\ldots,m_{2}$. 
Notice that the partial derivatives in the previous equalities are all constant functions.
Equalities \eqref{eq:antisymmcoeff} will be important in the proof of Proposition~\ref{prop:polynomial}.

Every polynomial on $\R^{n}$, thought of as equipped with dilations
$\delta_{r}$, is the sum of homogeneous polynomials and the maximum
among these degrees is the \emph{degree} \emph{of the polynomial}.
Polynomials of degree one are just affine functions $\ell$ of the
form 
\[
\ell(x)=\alpha+\langle\beta,\pi(x)\rangle
\]
with $\beta=(\beta_{1},\beta_{2},\ldots,\beta_{m})\in\R^{n}$, $\alpha\in\R$ and we have used
the projection
\begin{equation}\label{def:pi}
\pi:\R^n\to\R^m,\quad \pi(x)=(x_1,\ldots,x_m).
\end{equation}
 A homogeneous polynomial of degree two must have
the form
\begin{equation}
p(x)=\frac{1}{2}\sum_{i,j=1}^{m}c_{ij}x_{i}x_{j}+\sum_{l=m+1}^{m_{2}}c_{l}x_{l},\label{eq:2homogpolyn}
\end{equation}
where $c_{ij}$ and $c_{l}$ are real numbers, with $c_{ij}=c_{ji}$
for all $i,j=1,\ldots,m$. 
\begin{prop}
\label{prop:polynomial} Let $p:\R^{n}\to\R$ be a 2-homogeneous polynomial
of the form \eqref{eq:2homogpolyn} and let us consider the basis
$X_{m+1},\ldots,X_{m_{2}}$ of $V_{2}$. Then we have
\[
c_{ij}=\frac{1}{2}(X_{i}X_{j}p+X_{j}X_{i}p)\quad\text{and}\quad X_{i}X_{j}p=c_{ij}+\sum_{l=m+1}^{m_{2}}\gamma_{ij}^{l}c_{l},
\]
where $\gamma_{ij}^{l}$ are proportional to the structure constants
of the Lie algebra, namely 
\begin{equation}\label{eq:XiXjComm}
[X_{i},X_{j}]=\sum_{l=m+1}^{m_{2}}2\gamma_{ij}^{l}X_{l}
\end{equation}
and $i,j=1,\ldots,m$.
\end{prop}

\begin{proof}
We first define the symmetrized second order derivative 
\[
(X_{i}X_{j})^{s}\df\frac{X_{i}X_{j}+X_{j}X_{i}}{2},
\]
so that we can write
\begin{equation}\label{eq:XiXjsComm}
X_{i}X_{j}=(X_{i}X_{j})^{s}+\frac{1}{2}[X_{i},X_{j}],
\end{equation}
for every $i,j=1,\ldots,m$. 
Since $X_{i}X_{j}$ and $X_{l}$ are homogeneous differential operators of order $-2$ and $p$ has degree 2, the horizontal derivatives $X_{i}X_{j}p$ and $X_{l}p$ are constants.

By \eqref{eq:2homogpolyn} and \eqref{eq:vectfields}, we get
\[
X_{j}p=\sum_{i=1}^{m}c_{ji}x_{i}+\sum_{i=m+1}^{n}a_{ji}\der_{x_{i}}\left(\sum_{l=m+1}^{m_{2}}c_{l}x_{l}\right)=\sum_{i=1}^{m}c_{ji}x_{i}+\sum_{i=m+1}^{m_{2}}a_{ji}c_{i},
\]
for $j=1,\ldots,m$. As a consequence, taking into account the form of the vector fields \eqref{eq:vectfields}  and of the polynomial \eqref{eq:2homogpolyn},
for any $i,j=1,\ldots,m$
and $l=m+1,\ldots,m_{2}$, we get
\begin{equation}\label{eq:XiXj}
X_{i}X_{j}p=
c_{ij}+\sum_{s=m+1}^{m_{2}}\der_{x_{i}}a_{js}c_{s}.
\end{equation}
To establish the previous equality, we have also observed that 
the polynomials $a_{ji}$ are homogeneous of degree $d_i-d_j=1$, therefore they are only depending on their first $m$ variables. In particular, all the partial derivatives $\der_{x_l}a_{ji}$ are vanishing whenever the integers $l$ and $i$ take values from $m+1$ to $m_2$ and $j=1,\ldots,m$.
Combining \eqref{eq:XiXj} and \eqref{eq:antisymmcoeff}, we also obtain the first of the following equalities
\[
X_{i}X_{j}p+X_{j}X_{i}p=2c_{ij}\quad\text{and}\quad X_{l}p=c_{l},
\]
with $1\le i,j\le m$ and $m+1\le l\le m_2$. The latter directly follows from the form of \eqref{eq:2homogpolyn}. In conclusion, by virtue of \eqref{eq:XiXjComm}, \eqref{eq:XiXjsComm} and \eqref{eq:XiXj}, we have obtained that
\[
X_{i}X_{j}p=(X_{i}X_{j})^{s}p+\sum_{l=m+1}^{m_{2}}\gamma_{ij}^{l}X_{l}p=c_{ij}+\sum_{l=m+1}^{m_{2}}\gamma_{ij}^{l}c_{l},
\]
hence concluding the proof.
\end{proof}
\begin{defn}
\label{def:c_ijrx_0} For $B_{r}(x_{0})\Subset\Omega$ and $u\in W_{H,\text{loc}}^{2,1}(\Omega)$,
we define the matrix 
\[
P_{r}^{x_{0}}\df(D_{h}^{2}u)_{B_{r}(x_{0})}-\frac{1}{m}(\Delta_{H}u)_{B_{r}(x_{0})}I_{m}\in\R^{n\times n},
\]
where $I_{m}$ stands for the identity matrix and the $(i,j)$ entry
of $(D_{h}^{2}u)_{B_{r}(x_{0})}$ is the average $(X_{i}X_{j}u)_{B_{r}(x_{0})}$.
Associated to the ball $B_{r}(x_{0})$, we also define the coefficients
\[
c_{ij}^{r,x_{0}}\df\left(\frac{X_{i}X_{j}u+X_{j}X_{i}u}{2}\right)_{B_{r}(x_{0})}-\frac{1}{m}\dil{ij}\,(\Delta_{H}u)_{B_{r}(x_{0})}\quad\text{and}\quad c_{l}^{r,x_{0}}=(X_{l}u)_{B_{r}(x_{0})}.
\]
These numbers define the 2-homogeneous polynomial 
\[
p_{r}^{x_{0}}(x)=\frac{1}{2}\sum_{i,j=1}^{m}c_{ij}^{r,x_{0}}x_{i}x_{j}+\sum_{l=m+1}^{m_{2}}c_{l}^{r,x_{0}}x_{l},
\]
that we will show to be related to $P_{r}^{x_{0}}$. 
\end{defn}

\begin{cor}\label{c:2ndorderpolyn}
In the assumptions of Definition \ref{def:c_ijrx_0}, the 2-homogeneous
polynomial $p_{r}^{x_{0}}$ is harmonic and
\[
D_{h}^{2}p_{r}^{x_{0}}=P_{r}^{x_{0}}.
\]
\end{cor}

\begin{proof}
By Proposition \ref{prop:polynomial}, we have
\[
X_{i}X_{j}p_{r}^{x_{0}}=c_{ij}^{r,x_{0}}+\sum_{l=m+1}^{m_{2}}\gamma_{ij}^{l}c_{l}^{r,x_{0}},
\]
where $\gamma_{ii}^{l}=0$ and by definition of $c_{ii}^{r,x_{0}}$
we get 
\[
\Delta_{H}p_{r}^{x_{0}}=\sum_{i=1}^{m}c_{ii}^{r,x_{0}}=\sum_{i=1}^{m}\left[\left(X_{i}X_{i}u\right)_{B_{r}(x_{0})}-\frac{1}{m}\,(\Delta_{H}u)_{B_{r}(x_{0})}\right]=0.
\]
Finally, we observe that 
\[
\begin{split}X_{i}X_{j}p_{r}^{x_{0}}= & \left(\frac{X_{i}X_{j}u+X_{j}X_{i}u}{2}\right)_{B_{r}(x_{0})}-\frac{1}{m}\dil{ij}\,(\Delta_{H}u)_{B_{r}(x_{0})}+\left(\sum_{l=m+1}^{m_{2}}\gamma_{ij}^{l}X_{l}u\right)_{B_{r}(x_{0})}\\
= & \left(\frac{X_{i}X_{j}u+X_{j}X_{i}u}{2}\right)_{B_{r}(x_{0})}-\frac{1}{m}\dil{ij}\,(\Delta_{H}u)_{B_{r}(x_{0})}+\left(\frac{[X_{i},X_{j}]u}{2}\right)_{B_{r}(x_{0})}\\
= & \left(X_{i}X_{j}u\right)_{B_{r}(x_{0})}-\frac{1}{m}\dil{ij}\,(\Delta_{H}u)_{B_{r}(x_{0})},
\end{split}
\]
having taken into account that $[X_{i},X_{j}]=\sum_{l=m+1}^{m_{2}}2\gamma_{ij}^{l}X_{l}$
from Proposition \ref{prop:polynomial}.
\end{proof}
 The following $W_{H}^{2,p}$ estimates go back to the work of Folland
\cite{MR0494315}, see also the work by Bramanti and Brandolini
\cite{MR1608289} for more general hypoelliptic operators.  
\begin{thm}
[\cite{MR1608289}]\label{thm:BramEst} Let $1<p<\infty$ and consider two bounded open sets $\Omega$ and $\Omega'$ with $\Omega'\Subset\Omega$. Then there exists a constant $C>0$ such that for every $u\in W_{H}^{2,p}(\Omega)$ it holds
\begin{align}
\lVert X_{i}X_{j}u\rVert_{L^{p}(\Omega')} & \le C\left(\lVert\Delta_{H}u\rVert_{L^{p}(\Omega)}+\lVert u\rVert_{L^{p}(\Omega)}\right).\label{eq:W2P}
\end{align}
\end{thm}

It is well known that even for the classical Laplacian operator $\Delta$,
it is not true that $L^{\infty}$ bounds on $\Delta u$ imply the
boundedness of second order horizontal derivatives. Indeed our starting
point is that bounds on the $BMO$ norm of the sub-Laplacian $\Delta_{H}u$
show that the $BMO$ norm of the horizontal Hessian of $u$ is bounded,
according to the results of Bramanti et al. \cite{MR2174915}, \cite{MR3100057}.

\begin{thm}[{\cite[Theorem~2.10]{MR3100057}}]
\label{thm:psigmaBMO} Let $1<p<\infty,$ $0<\sigma<1$,
$u\in BMO^p_{\text{loc}}(B_1)$ and let $\Delta_{H}u\in BMO^{p}_{\text{loc}}(B_{1})$.
Then $X_iX_ju\in BMO^p(B_\sigma)$ for $i,j=1,\ldots,m$ and 
there exists a universal constant $C(\sigma,p)>0$ such that
\begin{align}\label{eq:BMO}
\lVert X_{i}X_{j}u\rVert_{BMO^{p}(B_{\sigma})} & \le C(\sigma,p)\left(\lVert\Delta_{H}u\rVert_{BMO_{\text{loc}}^{p}(B_{1})}+\lVert u\rVert_{BMO_{\text{loc}}^{p}(B_{1})}\right).
\end{align}
\end{thm}

\begin{rem}\label{rmk:BMOest}
Note that the nonvariational form of the operator in \cite[Theorem 2.10]{MR3100057}
needs a priori that the solution $u$ and its horizontal derivatives
are $BMO$. For our purposes, it is very important that the $BMO$ regularity
of $u$ is established with no a priori assumptions. This can be obtained
for the sub-Laplacian operator, since its distributional form allows us to apply a mollification argument. 
\end{rem}

In the sequel, we will also use the Frobenius norm $\lvert M\rvert$ for a matrix $M$ 
of coefficients $m_{ij}$, setting 
\[
\lvert M\rvert=\sqrt{\sum_{ij}\lvert m_{ij}\rvert^{2}}.
\]
With this definition we easily notice that $\lvert \Delta_{H}u\rvert\le \lvert D_{h}^{2}u\rvert$.

\begin{cor}\label{cor:BMOest}
Let $1<p<\infty$ and $0<\sigma<1$ be fixed. There exists 
$C(\sigma,p)>0$ such that for all $u\in BMO^p_{\text{loc}}(B_1)$ that satisfy the condition
$\Delta_Hu\in L^\infty(B_1)$, we have $X_iX_ju\in BMO^p(B_\sigma)$ 
for $i,j=1,\ldots,m$ and whenever $x_0\in B_\sigma$, $0<r<1-\sigma$, it holds
\begin{equation}\label{eq:u-Pr-BMO}
\fint_{B_{r}(x_{0})\cap B_{\sigma}}\lvert D_{h}^{2}u(y)-P_{r}^{x_{0}}\rvert dy\le C(\sigma,p)\left(\lVert\Delta_{\Heis}u\rVert_{L^{\infty}(B_{1})}+\lVert u\rVert_{BMO_{\text{loc}}^{p}(B_{1})}\right),
\end{equation}
where the matrix $P_{r}^{x_{0}}$ is introduced in Definition \ref{def:c_ijrx_0}.
\end{cor}

\begin{proof}
Theorem~\ref{thm:psigmaBMO} immediately implies that $X_iX_ju\in BMO^p(B_\sigma)$
and \eqref{eq:BMO} holds. Thus, we obtain the following estimates
\begin{align*}
 &\phantom{{}={}} \fint_{B_{r}(x_{0})\cap B_{\sigma}}\lvert D_{h}^{2}u(y)-P_{r}^{x_{0}}\rvert dy\\
 & \le\fint_{B_{r}(x_{0})\cap B_{\sigma}}\lvert D_{h}^{2}u(y)-(D_{h}^{2}u)_{B_{r}(x_{0})}\rvert dy+\frac{1}{m}|(\Delta_{\Heis}u)_{B_{r}(x_{0})}I_{m}\rvert\\
 & \le[D_{h}^{2}u]_{BMO(B_{\sigma})}+\frac{1}{\sqrt{m}}\|\Delta_{\Heis}u\|_{L^{\infty}(B_{1})}\\
 & \le\widetilde{C}(\sigma,p)\left(\lVert\Delta_{\Heis}u\rVert_{BMO_{\text{loc}}^{p}(B_{1})}+\lVert u\rVert_{BMO_{\text{loc}}^{p}(B_{1})}\right)+\lVert\Delta_{\Heis}u\rVert_{L^{\infty}(B_{1})}\\
 & \le C(\sigma,p)\left(\lVert\Delta_{\Heis}u\rVert_{L^{\infty}(B_{1})}+\lVert u\rVert_{BMO_{\text{loc}}^{p}(B_{1})}\right).
\end{align*}
This completes the proof.
\end{proof}

Heuristically, if $D_{h}^{2}u$ is not bounded around $x_{0}$, since
the difference of $D_{h}^{2}u$ and $P_{r}^{x_{0}}$ is controlled,
the $BMO$ estimate tells us that also $P_{r}^{x_{0}}$ becomes unbounded
as $r\to0^{+}$. Hence we will turn our attention to $P_{r}^{x_{0}}$.
In the following lemma, we will derive a general ``scaling estimate'' 
for the difference $\lvert P_{r_1}^{x_{0}}-P_{r_2}^{x_{0}}\rvert$.
In particular, when $r_2=2r_1$ we get a uniform bound on the growth of $|P_{r}^{x_{0}}|$ on dyadic scales.

\begin{lem}[Scaling estimates]\label{lem:dyadic-scale}
Let $1<p<\infty$ and $0<\lambda_1<1$ be fixed. Then there exists a universal constant $C(\lambda_1,p)>0$ such that for all $u\in BMO^p_{\text{loc}}(B_1)$ that satisfy the condition $\Delta_Hu\in L^\infty(B_1)$ the following holds. 
We have $X_iX_ju\in L^1_{\text{loc}}(B_1)$, where $i,j=1,\ldots,m$ and for $x_{0}\in B_{\lambda_{1}/3}$ the matrices of the form
\[
P_{r}^{x_{0}}\df(D_{h}^{2}u)_{B_{r}(x_{0})}-\frac{1}{m}(\Delta_{\Heis}u)_{B_{r}(x_{0})}I_{m},
\]
with $0<r_1<\min\{2\lambda_1/3,1-\lambda_1\}$ and $r_1<r_2<1-\lambda_1$, satisfy the following inequality
\[
\lvert P_{r_1}^{x_{0}}-P_{r_2}^{x_{0}}\rvert\le\left(\frac{r_2}{r_1}\right)^Q C(\lambda_1,p)\left(\lVert\Delta_{\Heis}u\rVert_{L^{\infty}(B_{1})}+\lVert u\rVert_{BMO^{p}_{\text{loc}}(B_{1})}\right).
\]
\end{lem}

\begin{proof}
Due to the $BMO$ estimate \eqref{eq:u-Pr-BMO} with $\sigma=\lambda_1$, we can estimate $\lvert P_{r_1}^{x_{0}}-P_{r_2}^{x_{0}}\rvert$ as follows
\begin{align}\label{eq:dyadic-scale}
 & \phantom{{}={}}\fint_{B_{r_1}(x_{0})\cap B_{\lambda_{1}}}\lvert P_{r_1}^{x_{0}}-P_{r_2}^{x_{0}}\rvert dx\nonumber \\
 & \le\fint_{B_{r_1}(x_{0})\cap B_{\lambda_{1}}}\lvert D_{h}^{2}u(y)-P_{r_1}^{x_{0}}\rvert dy+\fint_{B_{r_1}(x_{0})\cap B_{\lambda_{1}}}\lvert D_{h}^{2}u(y)-P_{r_2}^{x_{0}}\rvert dy\nonumber \\
 & \le\fint_{B_{r_1}(x_{0})\cap B_{\lambda_{1}}}\lvert D_{h}^{2}u(y)-P_{r_1}^{x_{0}}\rvert dy+\frac{|B_{r_2}(x_0)\cap B_{\lambda_1}|}{|B_{r_1}(x_0)\cap B_{\lambda_1}|}\fint_{B_{r_2}(x_{0})\cap B_{\lambda_{1}}}\lvert D_{h}^{2}u(y)-P_{r_2}^{x_{0}}\rvert dy\nonumber \\
 & \le C(\lambda_1,p)\left(1+\Big(\frac{r_2}{r_1}\right)^{Q}\Big)\left(\lVert\Delta_{\Heis}u\rVert_{L^{\infty}(B_{1})}+\lVert u\rVert_{BMO^{p}_{\text{ loc}}(B_{1})}\right).
\end{align}
The last inequality follows by taking into account our conditions on the radii $r_1$ and $r_2$.
Indeed, we have  
\begin{equation}
\frac{|B_{r_2}(x_0)\cap B_{\lambda_1}|}{|B_{r_1}(x_0)\cap B_{\lambda_1}|}\le\frac{|B_{r_2}(x_0)|}{|B_{r_1}(x_0)|}= \left(\frac{r_2}{r_1}\right)^{Q}.
\end{equation}
Finally, with a slight abuse of notation, we denote the constant $2C(\lambda_1,p)$
again by $C(\lambda_1,p)$ in the inequality of the lemma, concluding the proof. 
\end{proof}

\section{Proof of \texorpdfstring{$C^{1,1}_H$}{} regularity}\label{sec:C11}

This section represents the core of the paper. We establish the sub-quadratic
growth of the difference 
\[
u(y)-u(x_{0})-\langle\nabla_{h}u(x_{0}),\pi(x_{0}^{-1}y)\rangle-p_{r}^{x_{0}}(x_{0}^{-1}y)
\]
on the ball $B_{r}(x_{0})$, where $p_{r}^{x_{0}}$ is the harmonic
polynomial introduced in Definition \ref{def:c_ijrx_0}. We show that
when the norm of $D_{h}^{2}p_{r}^{x_{0}}$ is sufficiently large,
then the measure of the coincidence set $\{u=0\}$ decays in a quantitative way. This is one of the central facts, that leads us to the dichotomy argument of \cite{MR2999297} to reach the $C_{H}^{1,1}$ regularity. There are indeed
two cases: (i) when $|D_{h}^{2}p_{r}^{x_{0}}|$ is uniformly bounded as $r\to0^{+}$,
we immediately infer the regularity from the subquadratic growth,
(ii) if otherwise $|D_{h}^{2}p_{r}^{x_{0}}|$ grows without bound as $r\to0^{+}$,
then the coincidence set is ``small'' and we show that a suitable
adaptation of Caffarelli's polynomial iteration technique can lead
us to the $C_{H}^{1,1}$ regularity.

In the sequel, whenever we consider a function $u$ with essentially bounded sub-Laplacian $\Delta_Hu$,
then it is understood that $u$ is chosen to be of class $C^1_H$. 
The following remark rigorously justifies this convention.

\begin{rem}\label{rmk:bddC1}
Let $\Omega\subset \R^n$ be an open set and let $u:\Omega\to\R$ be a locally summable function
such that $\Delta_Hu\in L^\infty(\Omega)$.
From \cite[Theorem~6.1]{MR0494315}, we have $u\in W^{2,p'}_{H,\text{loc}}(\Omega)$ for every $p'>1$.
In view of Theorem~\ref{t:embed}, by standard arguments,
we can modify $u$ on a negligible set such that $u\in C^{1,\alpha}_{H}(\Omega')$ for any 
relatively compact open set $\Omega'\Subset\Omega$, where we have fixed some $p'>Q$
and $\alpha=1-\frac{Q}{p'}$. In particular, we have shown that, after the modification, $u\in C^1_H(\Omega)$.
\end{rem}

\begin{lem}
[Sub-quadratic growth]\label{lem:subquadratic-growth} Assume 
$u\in BMO_{\text{loc}}^p(B_1)$
such that $\Delta_{\Heis}u\in L^{\infty}(B_{1})$.
Let $\lambda,\sigma\in(0,1)$ and fix $p>1$. 
Then there exist $r_0>0$ and a universal constant $C(\lambda,\sigma,p)>0$, such that for any $x_{0}\in B_{\lambda}$ and $0<r\le r_{0}$, assuming that
\[
u(x_{0})=X_{i}u(x_{0})=0,\quad1\le i\le m
\]
and considering $p_{r}^{x_{0}}$, as given in Definition~\ref{def:c_ijrx_0}, the
following estimate holds
\begin{equation}
\sup_{y\in B_{\sigma r}(x_{0})}\lvert u(y)-p_{r}^{x_{0}}(x_{0}^{-1}y)\rvert\le 
C(\lambda,\sigma,p)\left(\lVert\Delta_{\Heis}u\rVert_{L^{\infty}(B_{1})}+\lVert u\rVert_{BMO_{\text{loc}}^{p}(B_{1})}\right)r^{2}.\label{eq:quadratic-growth}
\end{equation}
\end{lem}

\begin{proof}

We fix $x_0\in B_\lambda$ and $\lambda'=(1+\lambda)/2,$ so that for $0<r\le\lambda'-\lambda,$
we have the inclusion 
\begin{equation}\label{eq:Brlambda'}
B_r(x_{0})\subset B_{\lambda'}.
\end{equation}
Let us introduce the translated and rescaled function 
\[
u_{r,x_{0}}(x)\df\frac{u(x_{0}\dil rx)-p_{r}^{x_{0}}(\dil rx)}{r^{2}},
\]
observing that it is well defined in  $B_1$.
Taking into account that $u\in W^{2,p}_{H,\text{loc}}(B_1)$
and $\overline{B_r(x_0)}\subset \overline{B_{\lambda'}}\subset B_1$, then $u_{r,x_0}\in W^{2,p}_{H}(B_1)$. We are in the position to apply Corollary~\ref{cor:BMOest} to $u$ with  $\sigma=\lambda'$.
As a consequence of both Corollary~\ref{c:2ndorderpolyn} and  \eqref{eq:u-Pr-BMO}, taking into account \eqref{eq:Brlambda'}, it follows that
\begin{align} \label{eq:BMO-JN}
\lVert D_{h}^{2}u_{r,x_{0}}\rVert_{L^{1}(B_{1})} & =|B_1|\fint_{B_{1}}\lvert D_{h}^{2}u_{r,x_{0}}(x)\rvert\,dx\nonumber \\
 & =|B_1|\fint_{B_{r}(x_{0})}\lvert D_{h}^{2}u(y)-P_{r}^{x_{0}}\rvert dy \\
 & \le C(\lambda,p)\left(\lVert\Delta_{\Heis}u\rVert_{L^{\infty}(B_{1})}+\lVert u\rVert_{BMO_{\text{loc}}^{p}(B_{1})}\right) \nonumber.
\end{align}
Now we wish to apply the Poincar\'e inequality \eqref{eq:Poincare}
to $u_{r,x_{0}}-\ell_{r,x_{0}},$ where $\ell_{r,x_{0}}$ is an affine
function to be properly defined. If we let 
\[
\ell_{r,x_{0}}(x)\df(u_{r,x_{0}})_{B_{1}}+\langle(\nabla_{h}u_{r,x_{0}})_{B_{1}},\pi(x)\rangle,
\]
where $\pi(x)=(x_{1},\ldots,x_{m})$, it follows that
\[
\lVert u_{r,x_{0}}-\ell_{r,x_{0}}\rVert_{L^{1}(B_{1})}\le c\fint_{B_{1}}\Big|\nabla_{h}u_{r,x_{0}}-(\nabla_{h}u_{r,x_{0}})_{B_{1}}\Big|dx,
\]
since the average over $B_{1}$ of the linear part of $\ell_{r,x_{0}}$
is zero. Again, from the Poincar\'e inequality, using \eqref{eq:BMO-JN},
we get 
\begin{equation}
\begin{split}\lVert u_{r,x_{0}}-\ell_{r,x_{0}}\rVert_{L^{1}(B_{1})} & \le C\lVert D_{h}^{2}u_{r,x_{0}}\rVert_{L^{1}(B_{1})}\\
 & \le C(\lambda,p)\left(\lVert\Delta_{\Heis}u\rVert_{L^{\infty}(B_{1})}+\lVert u\rVert_{BMO_{\text{loc}}^{p}(B_{1})}\right).
\end{split}
\label{eq:L1uhat-bounded}
\end{equation}
For the sequel, we set $\hat{u}_{r,x_{0}}\df u_{r,x_{0}}-\ell_{r,x_{0}}$.
Since both $p_{r}^{x_{0}}$ and $\ell_{r,x_{0}}$ are harmonic, we
observe that 
\[
\Delta_{\Heis}\hat{u}_{r,x_{0}}(x)=(\Delta_{\Heis}u)(x_{0}\dil rx)=f(x_{0}\dil rx)
\]
for a.e.~$x\in B_{1}$, where we have set $f\df\Delta_Hu\in L^\infty(B_1)$. We set $g_{r,x_{0}}(x)=f(x_{0}\dil rx)\chi_{B_{1}}$
and we consider the decomposition $\hat{u}_{r,x_{0}}=\hat{v}_{r,x_{0}}+\hat{w}_{r,x_{0}}$,
where 
\[
\hat{v}_{r,x_{0}}=-g_{r,x_{0}}*\Gamma\quad\text{and}\quad\hat{w}_{r,x_{0}}=\hat{u}_{r,x_{0}}+g_{r,x_{0}}*\Gamma
\]
and $\Gamma$ is the fundamental solution for $\Delta_H$, introduced in
Definition~\ref{def:FundSol}. The explicit form of $\hat{v}_{r,x_{0}}$
allows us to get the estimate 
\[
|\hat{v}_{r,x_{0}}(x)|=\left|\int_{\mathbb{R}^{n}}\Gamma(z^{-1}x)g_{r,x_{0}}(z)dz\right|=\left|\int_{B_{1}}\Gamma(z^{-1}x)g_{r,x_{0}}(z)dz\right|\le C\lVert g_{r,x_{0}}\rVert_{L^{Q_{0}}(B_{1})}
\]
for every $x\in B_{1}$, where $Q_0=Q+1$ and $C>0$ can be seen as a universal constant. 
The previous estimate follows from the H\"older inequality, setting $q=Q_0/Q$ and taking into account the $(2-Q)$-homogeneity of $\Gamma$. Indeed, it holds
\begin{equation}\label{eq:Holderq}
\left|\int_{B_{1}}\Gamma(z^{-1}x)g_{r,x_{0}}(z)dz\right|\le 
\left(\int_{B_2} |\Gamma|^q\right)^{1/q}\|g_{r,x_0}\|_{L^{Q_0}(B_1)}
\end{equation}
for every $x\in B_1$.
As a consequence, we have proved that
\begin{equation}\label{eq:ABP-v-hat}
\lVert\hat{v}_{r,x_{0}}\rVert_{L^{\infty}(B_{1})}\le C\lVert\Delta_{\Heis}\hat{u}_{r,x_{0}}\rVert_{L^{Q_{0}}(B_{1})}.
\end{equation}
Since $\hat{w}_{r,x_{0}}$ is harmonic, from \cite[(5.52)]{MR2363343}
we have the mean value type formula
\[
\hat{w}_{r,x_{0}}(x)=\fint_{B^G_{(1-\sigma)/c_0}(x)}\Psi(x^{-1}z)\hat{w}_{r,x_{0}}(z)\,dz
\]
for any $x\in B_{\sigma}$, whenever $0<\sigma<1$ and with $c_0>1$ defined in \eqref{eq:C_0}.
We point out that the function $\Psi$ is 0-homogeneous with respect to dilations and smooth on $\mathbb{R}^{n}\sm\left\{ 0\right\}$, see \cite[Definition~5.5.1]{MR2363343} for more information.
Notice that with our assumptions we have the inclusion $B^G_{(1-\sigma)/c_0}(x)\subset B_1$. For every $x\in B_{\sigma}$, it holds
\begin{align*}
\lvert\hat{w}_{r,x_{0}}(x)\rvert & =\left|\fint_{B^G_{(1-\sigma)/c_0}(x)}\Psi(x^{-1}z)\hat{w}_{r,x_{0}}(z)\,dz\right|\\
 & \le\lVert\Psi\rVert_{L^{\infty}(B_{1})}\fint_{B^G_{(1-\sigma)/c_0}(x)}\left|\hat{w}_{r,x_{0}}(z)\right|dz\\
 & \le\lVert\Psi\rVert_{L^{\infty}(B_{1})}\frac{\lVert\hat{w}_{r,x_{0}}\rVert_{L^{1}(B_{1})}}{|B^G_{(1-\sigma)/c_0}(x)|}\le C(\sigma)\lVert\hat{w}_{r,x_{0}}\rVert_{L^{1}(B_{1})}.
\end{align*}
The constant $C(\sigma)$ only depends on $\sigma$ and it blows up
as $\sigma\to1^{-}.$ By the triangle inequality and \eqref{eq:ABP-v-hat}
we obtain that
\begin{equation}\label{eq:Linf-L1-w-hat}
\begin{split}
\lVert\hat{w}_{r,x_{0}}\rVert_{L^{\infty}(B_{\sigma})} & \le C(\sigma)\lVert\hat{w}_{r,x_{0}}\rVert_{L^{1}(B_{1})} \\
 & \le C(\sigma)\left(\lVert\hat{u}_{r,x_{0}}\rVert_{L^{1}(B_{1})}+\lVert\hat{v}_{r,x_{0}}\rVert_{L^{1}(B_{1})}\right) \\
 & \le C_{1}(\sigma)\left(\lVert\hat{u}_{r,x_{0}}\rVert_{L^{1}(B_{1})}+\lVert\Delta_{\Heis}\hat{u}_{r,x_{0}}\rVert_{L^{Q_{0}}(B_{1})}\right).
\end{split}
\end{equation}
We conclude from both \eqref{eq:ABP-v-hat} and \eqref{eq:Linf-L1-w-hat}
that 
\begin{equation}\label{eq:Linftyuhatrx_0}
\begin{split}
\lVert\hat{u}_{r,x_{0}}\rVert_{L^{\infty}(B_{\sigma})} & \le\lVert\hat{v}_{r,x_{0}}\rVert_{L^{\infty}(B_{\sigma})}+\lVert\hat{w}_{r,x_{0}}\rVert_{L^{\infty}(B_{\sigma})}\\
 & \le C_{2}(\sigma)\left(\lVert\hat{u}_{r,x_{0}}\rVert_{L^{1}(B_{1})}+\lVert\Delta_{\Heis}\hat{u}_{r,x_{0}}\rVert_{L^{Q_{0}}(B_{1})}\right). 
\end{split}
\end{equation}
Differentiating $\hat{v}_{r,x_{0}}$, seen as an integral, it turns
out that $\hat{v}_{r,x_{0}}\in C_{H}^{1}(B_{1})$.
Again arguing as in the proof of \eqref{eq:Holderq}, 
from the H\"older inequality and the $(1-Q)$-homogeneity of $X_j\Gamma$, we get the estimate
\begin{equation}
|X_{j}\hat{v}_{r,x_{0}}(x)|=\Big|\int_{B_{1}}X_{j}\Gamma(y^{-1}x)\,g_{r,x_{0}}(y)dy\Big|\le \overline{C}\|g_{r,x_{0}}\|_{L^{Q_{0}}(B_{1})}\label{eq:X_jGammaEstimate}
\end{equation}
for every $j=1,\ldots,m$, $x\in B_{1}$ and a fixed geometric
constant $\overline{C}>0.$ By Proposition~\ref{prop:derivEst}, we get a constant
$C_3(\sigma)>0$, such that 
\begin{equation}
\|\nabla_{h}\hat{w}_{r,x_{0}}\|_{L^{\infty}(B_{\sigma})}\le C_{3}(\sigma)\|\hat{w}_{r,x_{0}}\|_{L^{1}(B_{1})}.\label{eq:Nabla_Hwrx_0Estimate}
\end{equation}
Combining \eqref{eq:Linftyuhatrx_0}, \eqref{eq:X_jGammaEstimate} and \eqref{eq:Nabla_Hwrx_0Estimate}, along with the third inequality
of \eqref{eq:Linf-L1-w-hat}, we establish the first of the following
inequalities:
\begin{align*}
\lVert\hat{u}_{r,x_{0}}\rVert_{C_{H}^{1}(B_{\sigma})} & \le C_{4}(\sigma)\left(\lVert\hat{u}_{r,x_{0}}\rVert_{L^{1}(B_{1})}+\lVert\Delta_{H}\hat{u}_{r,x_{0}}\rVert_{L^{Q_{0}}(B_{1})}\right)\\
 & \le C_{5}(\sigma,p,\lambda)\left(\lVert\Delta_{\Heis}u\rVert_{L^{\infty}(B_{1})}+\lVert u\rVert_{BMO_{\text{loc}}^{p}(B_{1})}\right).
\end{align*}
The second inequality is a consequence of \eqref{eq:L1uhat-bounded}.
Since $u_{r,x_{0}}(0)=X_{j}u_{r,x_{0}}(0)=0$, by our assumptions
on $u$, and taking into account that $p_{r}^{x_{0}}(0)=X_{j}p_{r}^{x_{0}}(0)=0$,
we immediately infer from the $C_{H}^{1}$ estimate above that 
\begin{align*}
 & \phantom{{}={}}\lvert\ell_{r,x_{0}}(0)\rvert+\sum_{i=1}^{m}\lvert X_{i}\ell_{r,x_{0}}(0)\rvert\\
 & =\lvert\ell_{r,x_{0}}(0)-u_{r,x_{0}}(0)\rvert+\sum_{i=1}^{m}\lvert X_{i}\ell_{r,x_{0}}(0)-X_{i}u_{r,x_{0}}(0)\rvert\\
 & \le\lVert\hat{u}_{r,x_{0}}\rVert_{C^{1}(B_{\sigma})}\le C_{5}(\sigma,p,\lambda)\left(\lVert\Delta_{\Heis}u\rVert_{L^{\infty}(B_{1})}+\lVert u\rVert_{BMO_{\text{loc}}^{p}(B_{1})}\right).
\end{align*}
It follows that 
\[
\lVert\ell_{r,x_{0}}\rVert_{L^{\infty}(B_{\sigma})}\le C_{6}(\sigma,p,\lambda)\left(\lVert\Delta_{\Heis}u\rVert_{L^{\infty}(B_{1})}+\lVert u\rVert_{BMO_{\text{loc}}^{p}(B_{1})}\right).
\]
 As a consequence, it follows that
\begin{align}
\frac{1}{r^{2}}\sup_{y\in B_{\sigma r}(x_{0})}\lvert u(y)-p_{r}^{x_{0}}(x_{0}^{-1}y)\rvert & =\sup_{x\in B_{\sigma}}\left|\frac{u(x_{0}\dil rx)-p_{r}^{x_{0}}(\dil rx)}{r^{2}}\right|\nonumber \\
 & =\sup_{x\in B_{\sigma}}\lvert u_{r,x_{0}}(x)\rvert\nonumber \\
 & \le\sup_{x\in B_{\sigma}}\lvert\hat{u}_{r,x_{0}}(x)\rvert+\sup_{x\in B_{\sigma}}\lvert\ell_{r,x_{0}}(x)\rvert\nonumber \\
 & \le C(\sigma,p,\lambda)\left(\lVert\Delta_{\Heis}u\rVert_{L^{\infty}(B_{1})}+\lVert u\rVert_{BMO_{\text{loc}}^{p}(B_{1})}\right).\label{eq:utildebounded}
\end{align}
This finishes the proof.
\end{proof}
\begin{cor}
\label{cor:subquadraticnoZero} Assume $u\in L^\infty(B_1)$
such that $\Delta_{\Heis}u\in L^{\infty}(B_{1})$ and fix $0<\lambda,\sigma<1$. 
If we consider $\pi$ as in \eqref{def:pi}, then there exists $r_{0}>0$ such that the affine function
\[
\ell^{x_{0}}(z)\df u(x_{0})+\langle\nabla_{h}u(x_{0}),\pi(z)\rangle, \quad x_{0}\in B_{\lambda}
\]
satisfies the following properties. There exists a universal constant $C(\lambda,\sigma)>0$ such that for every $r\in(0,r_{0}]$ the following estimate holds
\begin{equation}
\sup_{y\in B_{\sigma r}(x_{0})}\lvert u(y)-\ell^{x_{0}}(x_{0}^{-1}y)-p_{r}^{x_{0}}(x_{0}^{-1}y)\rvert\le C(\lambda,\sigma)\left(\lVert\Delta_{\Heis}u\rVert_{L^{\infty}(B_{1})}+\lVert u\rVert_{L^{\infty}(B_{1})}\right)r^{2},\label{eq:quadratic-growth-1}
\end{equation}
where $p_{r}^{x_{0}}$ is as in Definition~\ref{def:c_ijrx_0}.
\end{cor}

\begin{proof}
Our assumptions allow us to apply Lemma~\ref{lem:subquadratic-growth} to  $y\to u(y)-\ell^{x_{0}}(x_0^{-1}y)$ with $p=2$. Then there exist $r_{0},C(\lambda,\sigma)>0$ such that 
\[
\sup_{y\in B_{\sigma r}(x_{0})}\lvert u(y)-\ell^{x_{0}}(x_0^{-1}y)-p_{r}^{x_{0}}(x_{0}^{-1}y)\rvert\le C(\lambda,\sigma)\left(\lVert\Delta_{\Heis}u\rVert_{L^{\infty}(B_{1})}+\lVert u-\ell^{x_{0}}\rVert_{BMO_{\text{loc}}^{2}(B_{1})}\right)r^{2}
\]
for every $r\in(0,r_{0}].$ In addition, we have
\[
\lVert u-\ell^{x_{0}}\rVert_{BMO_{\text{loc}}^{2}(B_{1})}\le C\|u-\ell^{x_{0}}\|_{L^{\infty}(B_{1})}\le C'(\|u\|_{L^{\infty}(B_{1})}+|\nabla_{h}u(x_{0})|).
\]
We set $f=\Delta_{H}u\in L^{\infty}(B_{1})$ and write $v=f*\Gamma$, getting
\[
|\nabla_{h}u(x_{0})|\le|\nabla_{h}(u+v)(x_{0})|+|\nabla_{h}(f*\Gamma)(x_{0})|.
\]
Arguing as in \cite[Lemma 4.1]{gilbarg:01}, we establish
\begin{equation}
|\nabla_hv(x_0)|=|\nabla_{h}(f*\Gamma)(x_{0})|\le\|\Delta_Hu\|_{L^{\infty}(B_{1})}\|\nabla_{h}\Gamma\|_{L^{1}(B_{2})},
\end{equation}
therefore we have
\[
|\nabla_{h}u(x_{0})|\le|\nabla_{h}(u+v)(x_{0})|+\|\Delta_Hu\|_{L^\infty(B_1)} \|\nabla_h\Gamma\|_{L^1(B_2)}.
\]
Since $u+v$ is harmonic in $B_1$, by \eqref{eq:Xalpha}, it follows that
\[
\begin{split}
|\nabla_{h}(u+v)(x_{0})|&\le C_0(\|u\|_{L^\infty(B_1)}+\|v\|_{L^\infty(B_1)}) \\
&\le C_0(\|u\|_{L^{\infty}(B_{1})}+\|\Delta_Hu\|_{L^{\infty}(B_{1})}\|\Gamma\|_{L^1(B_2)}).
\end{split}
\]
This immediately leads us to our claim.
\end{proof}

\begin{rem}
\label{rem:subquadraticnoZero} Notice that under the same assumptions
of Corollary \ref{cor:subquadraticnoZero}, we can assume that for
every $\lambda,\sigma\in(0,1)$ and any $x_{0}\in B_{\lambda}$, there
exist $\tilde r_{0}>0$ and $C>0$, only depending on $\lambda$ and $\sigma$,
such that for all $r\in(0,\tilde r_{0}]$ the following estimate holds
\begin{equation}
	\sup_{y\in B_{\sigma_0r}^{G}(x_{0})}\lvert u(y)-\ell^{x_{0}}(x_{0}^{-1}y)-p_{r}^{x_{0}}(x_{0}^{-1}y)\rvert\le C(\lambda,\sigma)\left(\lVert\Delta_{\Heis}u\rVert_{L^{\infty}(B_{1})}+\lVert u\rVert_{L^{\infty}(B_{1})}\right)r^{2},
\end{equation}
for $\sigma_0=\sigma/c_{0}\in(0,1/c_{0})$ and additionally we have the
inclusion $\overline{B^G_{r_0}(x_0)}\subset B_1$.
This is a consequence of the definition of $c_{0}$ in \eqref{eq:defc_0}.
If we set $r_0\df \tilde r_0/c_0$, replacing $r$ by $c_0r$, 
we may rephrase the previous estimate as follows
\begin{equation}\label{eq:quadratic-growthBG}
	\sup_{y\in B_{\sigma r}^{G}(x_{0})}\lvert u(y)-\ell^{x_{0}}(x_{0}^{-1}y)-p_{c_0r}^{x_{0}}(x_{0}^{-1}y)\rvert\le C(\lambda,\sigma)c_0^2\left(\lVert\Delta_{\Heis}u\rVert_{L^{\infty}(B_{1})}+\lVert u\rVert_{L^{\infty}(B_{1})}\right)r^{2}
\end{equation}
for every $0<r\le r_0$.
\end{rem}

We introduce now the important definition of \emph{coincidence set}:
\[
\Lambda\df\{x\in B_{1}:u(x)=0\}.
\]
We will perform a blow-up of $\Lambda$ around a fixed point $x_0\in B_{1/2}$,
considering the rescaled and translated coincidence sets
\begin{equation}
\Lambda_{r}(x_{0})\df\{x\in B_{1}^{G}:u(x_{0}\dil rx)=0\}, \label{eq:Lambda_r}
\end{equation}
for $0<r\le r_0$ and some $r_0>0$ such that $B_r^G(x_0)\subset B_1$.
Notice that in the previous definition the gauge distance is used for
technical reasons, related to the existence of
solutions to the Dirichlet problem with respect to the sub-Laplacian.

The next result is a technical lemma, that will be used both
to get the decay estimates in Proposition~\ref{prop:decay-of-coincidence-set} and
to establish the regularity in Theorem~\ref{thm:regularity}.

\begin{lem}\label{lem:blowup-approximation}
Let $f$ be such that $f*\Gamma\in C_{H}^{1,1}(B_{1})$ and let $u$ solve
\eqref{eq:no-sign-obstacle-problem-Carnot} in $B_1$. Then for every $0<\lambda,\sigma<1$, there exists $r_0>0$ such that for every $x_0\in B_\lambda$ 
we have $\overline{B_{r_{0}}^{G}(x_{0})}\subset B_{1}$ and the following holds.
Let us consider the translated and rescaled function 
\begin{equation}
u_{r,x_{0}}(x)\df\frac{u(x_{0}\dil rx)-\ell^{x_{0}}(\delta_{r}x)-p_{c_0r}^{x_{0}}(\dil rx)}{r^{2}},\label{eq:u_rxp_r1}
\end{equation}
where $p_{c_0r}^{x_{0}}$ is introduced in Definition \ref{def:c_ijrx_0}
and $\ell^{x_{0}}(z)=u(x_{0})+\langle\nabla_{h}u(x_{0}),\pi(z)\rangle$.
For each $r\in(0,r_{0}]$ we also define $v_{r,x_{0}}$ as the solution to
\begin{equation}\label{eq:Dirichdelta1}
\begin{cases}
\Delta_{H}v_{r,x_{0}}=f_{r,x_{0}}, & \text{in }B_{\sigma}^{G},\\
v_{r,x_{0}}=u_{r,x_{0}}, & \text{on }\partial B_{\sigma}^{G},
\end{cases}
\end{equation}
where $f_{r,x_{0}}(x)=f(x_{0}\delta_{r}x)\chi_{B^G_\sigma}$.
Then there exists a universal constant $C(\lambda,\sigma)>0$, depending on $\lambda$ and $\sigma$, such that
\begin{equation}\label{eq:D2v_r,x0-estshort1}
\|D_{h}^{2}v_{r,x_{0}}\|_{L^{\infty}(B_{\sigma^2}^{G})}\le C(\lambda,\sigma)(\|D_{h}^{2}(f*\Gamma)\|_{L^{\infty}(B_{1})}+\lVert u\rVert_{L^{\infty}(B_{1})}).
\end{equation}
\end{lem}
\begin{proof}
Due to Remark~\ref{rem:subquadraticnoZero}, there exists $r_0>0$ such that
for every $x_0\in B_\lambda$ we have $\overline{B_{r_{0}}^{G}(x_{0})}\subset B_{1}$ and
\eqref{eq:quadratic-growthBG} holds for every $r\in(0,r_0]$.
We write the solution to the Dirichlet problem \eqref{eq:Dirichdelta1} in the form
\[
v_{r,x_{0}}=\eta_{r,x_{0}}+\zeta_{r,x_{0}},
\]
where $\zeta_{r,x_{0}}$ solves 
\[
\begin{cases}
\Delta_{H}\zeta_{r,x_{0}}=0, & \text{in }B_{\sigma}^{G},\\
\zeta_{r,x_{0}}=u_{r,x_{0}}-\eta_{r,x_{0}}, & \text{on }\partial B_{\sigma}^{G}
\end{cases}
\]
and we have defined 
\[
\eta_{r,x_{0}}=-f_{r,x_{0}}*\Gamma.
\]
Indeed, the open set $B_{\sigma}^{G}$ is regular with respect to $\Delta_{H},$
see \cite[Proposition 7.2.8]{MR2363343}. From the identity
\[
D_{h}^{2}v_{r,x_{0}}=-D_{h}^{2}(f_{r,x_{0}}*\Gamma)+D_{h}^{2}\zeta_{r,x_{0}},
\]
taking into account the equality $D^2_h(f*\Gamma)(x_0\delta_rx)=D^2_h(f_{r,x_0}*\Gamma)(x)$ for a.e. $x\in B^G_{\sigma^2}$ and the estimate \eqref{eq:Xalpha} we obtain that
\begin{equation}
\begin{split}\|D_{h}^{2}v_{r,x_{0}}\|_{L^{\infty}(B_{\sigma^2}^{G})} & \le\|D_{h}^{2}(f_{r,x_{0}}*\Gamma)\|_{L^{\infty}(B_{\sigma^2}^{G})}+\|D_{h}^{2}\zeta_{r,x_{0}}\|_{L^{\infty}(B_{\sigma^2}^{G})}\\
& \le\|D_{h}^{2}(f*\Gamma)\|_{L^{\infty}(B_{\sigma^2r}^{G}(x_{0}))}+C(\sigma)\|\zeta_{r,x_{0}}\|_{L^{\infty}(B_{\sigma}^{G})}.
\end{split}
\end{equation}
Now we combine the maximum principle and the Dirichlet problem \eqref{eq:Dirichdelta1} to get
\begin{equation}
\|D_{h}^{2}v_{r,x_{0}}\|_{L^{\infty}(B_{\sigma^2}^{G})}
\le\|D_{h}^{2}(f*\Gamma)\|_{L^{\infty}(B_{1})}+C(\sigma)\|u_{r,x_{0}}+f_{r,x_{0}}*\Gamma\|_{L^{\infty}(\der B_{\sigma}^{G})}.
\end{equation}
Due to the version of the sub-quadratic growth in  \eqref{eq:quadratic-growthBG},  taking into account the definition \eqref{eq:u_rxp_r1} and the immediate estimate $$\|f_{r,x_0}*\Gamma\|_{L^\infty(B^G_\sigma)}\le C \|f\|_{L^\infty(B_1)},$$
where $C>0$ only depends on $\Gamma$, it follows that
\begin{equation}
\|D_{h}^{2}v_{r,x_{0}}\|_{L^{\infty}(B_{\sigma^2}^{G})}\le\|D_{h}^{2}(f*\Gamma)\|_{L^{\infty}(B_{1})}+C(\lambda,\sigma)(\|u\rVert_{L^{\infty}(B_{1})}+\lVert f\|_{L^{\infty}(B_{1})}).
\end{equation}
In conclusion, we have established the following estimate
\begin{equation}
\|D_{h}^{2}v_{r,x_{0}}\|_{L^{\infty}(B_{\sigma^2}^{G})}\le C(\lambda,\sigma)(\|D_{h}^{2}(f*\Gamma)\|_{L^{\infty}(B_{1})}+\lVert u\rVert_{L^{\infty}(B_{1})}),\label{eq:D2v_r,x0-estshort}
\end{equation}	
concluding the proof.
\end{proof}

\begin{prop}[Decay of the coincidence set]\label{prop:decay-of-coincidence-set} Let
$f$ be such that $f*\Gamma\in C_{H}^{1,1}(B_{1})$ and let $u$ solve
\eqref{eq:no-sign-obstacle-problem-Carnot}. Then for every $\beta>0,$
there exist $C_{\beta}>0$ and $r_{0}>0$ so that if $0<r\le r_{0}$,
$x_{0}\in B_{1/2}$, and the estimate
\[
|P_{r}^{x_{0}}|\ge C_{\beta}\left(\lVert D_{h}^{2}(f*\Gamma)\rVert_{L^{\infty}(B_{1})}+\lVert u\rVert_{L^{\infty}(B_{1})}\right)
\]
holds, we have
\begin{equation}\label{eq:Decayr}
\lvert\Lambda_{r/2}(x_{0})\rvert\le\frac{\lvert\Lambda_{r}(x_{0})\rvert}{2^{\beta Q}}.
\end{equation}
\end{prop}
\begin{proof}
Our assumptions allow us to apply Lemma~\ref{lem:blowup-approximation} with $\lambda=1/2$,
where we choose $\sigma\in[1/\sqrt{2},1)$.
Let $r_0>0$,  $v_{r,x_{0}}$ and $u_{r,x_{0}}$ be as in the same lemma and define 
\[
w_{r,x_{0}}\df v_{r,x_{0}}-u_{r,x_{0}}.
\] 
Lemma~\ref{lem:blowup-approximation} yields a constant $C(\sigma)>0$ such that
\begin{equation}\label{eq:D2v_r,x0-estshort2}
\|D_{h}^{2}v_{r,x_{0}}\|_{L^{\infty}(B_{\sigma^2}^{G})}\le C(\sigma)(\|D_{h}^{2}(f*\Gamma)\|_{L^{\infty}(B_{1})}+\lVert u\rVert_{L^{\infty}(B_{1})})
\end{equation}
for every $r\in(0,r_0]$. In addition, from the definition of $w_{r,x_0}$ we observe that
\[
\begin{cases}
\Delta_{\Heis}w_{r,x_{0}}=f_{r,x_{0}}\chi_{\Lambda_{r}(x_{0})} & \text{in }B_{\sigma}^{G},\\
w_{r,x_{0}}=0 & \text{on }\partial B_{\sigma}^{G}.
\end{cases}
\]
By uniqueness, it follows that
\[
w_{r,x_{0}}=-\big(f_{r,x_{0}}\chi_{\Lambda_{r}(x_{0})}\big)*G_{B_{\sigma}^{G}}
\]
where $G_{B_{\sigma}^{G}}$ is the Green function of $B_\sigma^{G}$, according to \cite[Definition 9.2.1]{MR2363343}. 
From the definition of Green function, we have $G_{B_\sigma^{G}}\ge0$.
In addition, taking into account \cite[Proposition~9.2.12(iv)]{MR2363343}, we also notice
that the maximum principle gives \[
G_{B_\sigma^{G}}(x,y)\le\Gamma(x^{-1}y)
\]
for every $x,y\in B^G_\sigma$ with $x\neq y$.
Then a standard
convolution estimate yields 
\begin{equation}\label{eq:wconv}
\lVert w_{r,x_{0}}\rVert_{L^{\infty}(B_\sigma^{G})}\le C\lVert f\rVert_{L^{\infty}(B_{r\sigma}^{G}(x_{0}))}\lVert\chi_{\Lambda_{r}(x_{0})}\rVert_{L^{Q}(B_{1})}\le C\lVert f\rVert_{L^{\infty}(B_{1})}\lvert\Lambda_{r}(x_{0})\rvert^{1/Q}
\end{equation}
for some geometric constant $C>0.$ The $W_{H}^{2,p}$ estimates
\eqref{eq:W2P} give a universal constant $C_1$, depending on $\sigma$, such that
\begin{align*}
\int_{B_{\sigma^2}^{G}}\lvert D_{h}^{2}w_{r,x_{0}}(x)\rvert^{2Q}\,dx & \le C_{1}\left(\lVert f_{r,x_{0}}\chi_{\Lambda_{r}(x_{0})}\rVert_{L^{2Q}(B_{\sigma}^{G})}+\lVert w_{r,x_{0}}\rVert_{L^{2Q}(B_\sigma^{G})}\right)^{2Q}\\
 & \le C_{2}\lVert f\rVert_{L^{\infty}(B_{1})}^{2Q}(\lvert\Lambda_{r}(x_{0})\rvert+\lvert\Lambda_{r}(x_{0})\rvert^{2})\\
 & \le C_{3}\lVert f\rVert_{L^{\infty}(B_{1})}^{2Q}\lvert\Lambda_{r}(x_{0})\rvert.
\end{align*}
The second inequality is again a consequence of a convolution estimate, joined with \eqref{eq:wconv}.
Since $|\Lambda_r(x_0)|\le |B^G_1|$, the third inequality is also established. 
Furthermore, taking the second order horizontal derivatives in the
definition \eqref{eq:u_rxp_r1}, we get the equality 
\[
P_{c_0r}^{x_{0}}=(D_{h}^{2}u)(x_{0}\dil rx)+D_{h}^{2}w_{r,x_{0}}(x)-D_{h}^{2}v_{r,x_{0}}(x).
\]
and also $\Lambda_{r\sigma^2}(x_{0})=\dil {\sigma^{-2}}(\Lambda_{r}(x_{0})\cap B_{\sigma^2}^{G}).$
 In addition, arguing as in \cite[Lemma 7.7]{gilbarg:01}, we can
establish that $(D^{2}_hu)(x_{0}\dil rx)=0$ a.e. on the coincidence
set $\Lambda_{r}(x_{0})$. Taking into account all previous facts,
we get
\begin{align*}
 & \phantom{{}={}}\sigma^{2Q}\lvert\Lambda_{r\sigma^2}(x_{0})\rvert\lvert P_{c_0r}^{x_{0}}\rvert^{2Q}\\
 & =\lvert\Lambda_{r}(x_{0})\cap B_{\sigma^2}^{G}\rvert\lvert P_{c_0r}^{x_{0}}\rvert^{2Q}\\
 & =\int_{\Lambda_{r}(x_{0})\cap B_{\sigma^2}^{G}}\lvert P_{c_0r}^{x_{0}}\rvert^{2Q}\,dx\\
 & =\int_{\Lambda_{r}(x_{0})\cap B_{\sigma^2}^{G}}\lvert(D_{h}^{2}u)(x_{0}\dil rx)+D_{h}^{2}w_{r,x_{0}}(x)-D_{h}^{2}v_{r,x_{0}}(x)\rvert^{2Q}\,dx\\
 & =\int_{\Lambda_{r}(x_{0})\cap B_{\sigma^2}^{G}}\lvert D_{h}^{2}w_{r,x_{0}}(x)-D_{h}^{2}v_{r,x_{0}}(x)\rvert^{2Q}\,dx\\
 & \le4^{Q}\int_{\Lambda_{r}(x_{0})\cap B_{\sigma^2}^{G}}\lvert D_{h}^{2}w_{r,x_{0}}(x)\rvert^{2Q}+\lvert D_{h}^{2}v_{r,x_{0}}(x)\rvert^{2Q}\,dx\\
 & \le C_{2}(\sigma)\left(\lVert f\rVert_{L^{\infty}(B_{1})}^{2Q}\lvert\Lambda_{r}(x_{0})\rvert+\lvert\Lambda_{r\sigma^2}(x_{0})\rvert\left(\lVert D_{h}^{2}(f*\Gamma)\rVert_{L^{\infty}(B_{1})}+\lVert u\rVert_{L^{\infty}(B_{1})}\right)^{2Q}\right).
\end{align*}
Consequently, 
\[
\frac{\sigma^{2Q}\lvert P_{c_0r}^{x_{0}}\rvert^{2Q}-C_2(\sigma)\left(\lVert D_{h}^{2}(f*\Gamma)\rVert_{L^{\infty}(B_{1})}+\lVert u\rVert_{L^{\infty}(B_{1})}\right)^{2Q}}{C_2(\sigma)\lVert f\rVert_{L^{\infty}(B_{1})}^{2Q}}\lvert\Lambda_{r\sigma^2}(x_{0})\rvert\le\lvert\Lambda_{r}(x_{0})\rvert.
\]
We see that the coefficient in front of $\lvert\Lambda_{r\sigma^2}(x_{0})\rvert$
is bigger than $2^{\beta Q}$ if 
\begin{equation}\label{eq:sigmaEst}
\sigma^{2Q}\lvert P_{c_0r}^{x_{0}}\rvert^{2Q}\ge C_2(\sigma)2^{\beta Q}\lVert f\rVert_{L^{\infty}(B_{1})}^{2Q}+C_2(\sigma)\left(\lVert D_{h}^{2}(f*\Gamma)\rVert_{L^{\infty}(B_{1})}+\lVert u\rVert_{L^{\infty}(B_{1})}\right)^{2Q}.
\end{equation}
By the simple inequality $\lVert D_{h}^{2}(f*\Gamma)\rVert_{L^{\infty}(B_{1})}\ge\lVert f\rVert_{L^{\infty}(B_{1})}$, a few more computations lead us to the following sufficient condition 
\[
\lvert P_{c_0r}^{x_{0}}\rvert\ge\left[ C_2(\sigma)\sigma^{-2Q}(2^{\beta Q}+2^{2Q-1})\right]^{1/2Q}\left(\lVert D_{h}^{2}(f*\Gamma)\rVert_{L^{\infty}(B_{1})}+\lVert u\rVert_{L^{\infty}(B_{1})}\right),
\]
to get \eqref{eq:sigmaEst} to hold. Finally, we choose $\sigma=1/\sqrt{2}$, then the proof follows by choosing the constant $C_{\beta}$ in our statement equal to $\sqrt[2Q]{C_2(1/\sqrt{2})2^{Q}(2^{\beta Q}+2^{2Q-1})}$
and replacing $c_0r$ by $r$.
\end{proof}
To carry out the proof of the $C_{H}^{1,1}$ regularity, we need a
Calderón type second order differentiability, according to the next
definition. 
\begin{defn}
We say that $u\in L_{\text{loc}}^{1}(\Omega)$ is \emph{twice $L^{1}$
differentiable at $x_{0}$} if there exists a polynomial $t$ of degree
less than or equal to two, such that 
\[
\frac{1}{r^{2}}\fint_{_{B_{r}(x_{0})}}|u(z)-t(z)|dz\to0\quad\text{as}\quad r\to0^+.
\]
The polynomial $t$ has the following form 
\[
t(x)=c_{0}+\sum_{l=1}^{m}v_{l}(x_{l}-x_{0l})+\frac{1}{2}\sum_{i,j=1}^{m}c_{ij}(x_{i}-x_{0i})(x_{j}-x_{0j})+\sum_{l=m+1}^{m_{2}}c_{l}(x_{l}-x_{0l}),
\]
$x=(x_{1},\ldots,x_{n})$, $x_{0}=(x_{01},\ldots,x_{0n})$ and $c_{0},c_{ij},c_{l}\in\R.$
\end{defn}

It is possible to show that any $u\in W_{H,\text{loc}}^{2,1}(\Omega)$
is twice $L^{1}$ differentiable a.e.\ in $\Omega$. Furthermore, if
the function is twice $L^{1}$ differentiable at a Lebesgue point
$x_{0}\in\Omega$ of all functions $X_{i}X_{j}u,$ $X_{j}u$ and $u$,
then the corresponding polynomial is unique and it has the following
form
\begin{align*}
u(x_{0})+\sum_{j=1}^{m}X_{j}u(x_{0})\,(x_{j}-x_{0j}) & +\frac{1}{2}\sum_{i,j=1}^{m}((X_{i}X_{j}+X_{j}X_{i})u)(x_{0})\,(x_{i}-x_{0i})(x_{j}-x_{0j})\\
 & +\sum_{l=m+1}^{m_{2}}X_{l}u(x_{0})\,(x_{l}-x_{0l}),
\end{align*}
see \cite{MR2146126} for more information. We are now in the position to prove the optimal interior regularity of solutions to the no-sign obstacle-type problem \eqref{eq:no-sign-obstacle-problem-Carnot}. 
\begin{thm}[$C^{1,1}$ regularity] \label{thm:regularity} Let $u\in L^{\infty}(B_{1})$ be a distributional solution to \eqref{eq:no-sign-obstacle-problem-Carnot} in the unit ball $B_1$. Let $f:B_{1}\to\R$ be locally summable such that $f*\Gamma\in C_{H}^{1,1}(B_{1})$.
Then there exists a universal constant $C>0$ such that, after a modification on a negligible set, we have $u\in C_{H}^{1,1}(B_{1/4})$ and
\[
\lVert D_{h}^{2}u\rVert_{L^{\infty}(B_{1/4})}\le C\left(\lVert D_{h}^{2}(f*\Gamma)\rVert_{L^{\infty}(B_{1})}+\lVert u\rVert_{L^{\infty}(B_{1})}\right).
\]
\end{thm}

\begin{proof}
We consider $C_\beta$ as in Proposition~\ref{prop:decay-of-coincidence-set} and fix $\beta=4$.
We consider a priori the following constant
\[
K=C_4\left(\|D_{h}^{2}(f*\Gamma)\|_{L^{\infty}(B_{1})}+\|u\|_{L^{\infty}(B_{1})}\right).
\]
Combining the Hölder inequality and Theorem \ref{thm:BramEst}, 
taking into account that the constant $C_\beta$ in Proposition~\ref{prop:decay-of-coincidence-set} is bounded from below by a universal positive constant independent of $\beta$, we can find a universal constant $C_{1}'>0$ such that \begin{equation}\label{eq:D2hEstim}
\|D_{h}^{2}u\|_{L^{1}(B_{1/2})}\le C'_{1}K.
\end{equation}
Let $r_0>0$ be the minimum among the $r_0$'s of Remark~\ref{rem:subquadraticnoZero},
Lemma~\ref{lem:blowup-approximation} with $\lambda=1/4$ and Proposition~\ref{prop:decay-of-coincidence-set} with $\beta=4$. 
We fix an integer $i_0$ such that
\begin{equation}\label{eq:i_0}
i_0\ge 3+\log_2c_0,
\end{equation}
such that
$2^{-i_0}\le r_0$, where $c_0$ is the geometric constant appearing
in \eqref{eq:C_0}. Then \eqref{eq:D2hEstim} provides us with a universal constant
$\overline{C}_{1}\ge1$ such that
\begin{equation}
|P_{2^{-i_0}}^{y}|\le \overline{C}_{1}K\label{eq:C_1K}
\end{equation}
for all $y\in B_{1/4}$. Notice that $\overline{C}_1$ actually depends on $i_0$. However, this integer is fixed throughout the proof. 
We have chosen $i_0$ to satisfy also \eqref{eq:i_0} in view of the subsequent application of Lemma~\ref{lem:dyadic-scale} with $\lambda_1=3/4$. 
We can fix $x_{0}\in B_{1/4}$ such that $u$ is twice
$L^{1}$ differentiable at $x_{0}.$ 
Using \cite[Theorem 3.8]{MR2146126}
for $p=1$ and $k=2$, the set of these differentiability points has
full measure in $B_{1}$. We can further write $u=v-w$, such
that
\[
\Delta_{H}v=f\quad\text{and}\quad\Delta_{H}w=f\chi_{\Lambda}
\]
on $B_{1},$ where $v=-f*\Gamma.$ By assumption $v\in C_{H}^{1,1}(B_{1})$,
hence it is also a.e.\ twice $L^{1}$ differentiable, therefore we
can further assume $v$ is twice $L^{1}$ differentiable at $x_{0}$,
having the set of these points full measure in $B_{1}$. Now, only
two cases may occur. 

\noindent
{\bf Case 1}: $\displaystyle\boldsymbol{\liminf_{k\to \infty}\lvert P_{2^{-k}}^{x_0}\rvert\le \overline{C}_1K}$. At
our point $x_{0}$, we have
\begin{align*}
\lvert D_{h}^{2}u(x_{0})\rvert & =\Big|\lim_{k\to\infty}\fint_{B_{2^{-k}}(x_{0})}D_{h}^{2}u(y)\,dy\Big|\\
 & =\lim_{k\to\infty}\Big|\Big(P_{2^{-k}}^{x_{0}}+\frac{(\Delta_{H}u)_{B_{2^{-k}}(x_{0})}}{m}I_{m}\Big)\Big|\\
 & \le\liminf_{k\to\infty}\Big(|P_{2^{-k}}^{x_{0}}|+\frac{1}{\sqrt{m}}|(\Delta_{H}u)_{B_{2^{-k}}(x_{0})}|\Big)\\
 & =\frac{1}{\sqrt{m}}|\Delta_{H}u(x_{0})|+\liminf_{k\to\infty}|P_{2^{-k}}^{x_{0}}|\\
 & \le\frac{1}{\sqrt{m}}\|f\|_{L^{\infty}(B_{1})}+\overline{C}_{1}K\\
 & \le\|D_{h}^{2}(f*\Gamma)\|_{L^{\infty}(B_{1})}+\overline{C}_{1}K.
\end{align*}
Therefore
\[
\lvert D_{h}^{2}u(x_{0})\rvert\le(\overline{C}_{1}C_{4}+1)(\lVert D_{h}^{2}(f*\Gamma)\rVert_{L^{\infty}(B_{1})}+\lVert u\rVert_{L^{\infty}(B_{1})}).
\]
{\bf Case 2}: $\displaystyle\boldsymbol{\liminf_{k\to \infty}\lvert P_{2^{-k}}^{x_0}\rvert> \overline{C}_1K}$.  Then the following integer is well defined
\[
k_{0}\df\min\{k\in\N:k\ge i_{0},\,\lvert P_{2^{-j}}^{x_{0}}\rvert>\overline{C}_{1}K,\text{ for all }j\ge k\}.
\]
The positive integer $k_{0}$ possibly depends on $x_{0}$. 
We notice that from the definition of $k_0$, we have $\lvert P_{2^{-k_{0}+1}}^{x_{0}}\rvert\le \overline{C}_{1}K$. 
The strict inequality $k_{0}>i_{0}$ follows by \eqref{eq:C_1K}.
In view of our choice of $i_0$, that satisfies \eqref{eq:i_0} and then $i_0>3$, we can apply Lemma~\ref{lem:dyadic-scale} with $\lambda_1=3/4$. Indeed, we have $B_{1/4}=B_{\lambda_1/3}$ and 
\[
2^{-k_0+1}<\min\left\{\frac23\lambda_1,1-\lambda_{1}\right\}=2^{-2},
\]
so Lemma~\ref{lem:dyadic-scale} with $r_1=2^{-k_0}$ and $r_2=2^{-k_0+1}$ yields
\begin{equation}\label{eq:PC_0}
\begin{split}
\lvert P_{2^{-k_{0}}}^{x_{0}}\rvert & \le\lvert P_{2^{-k_{0}+1}}^{x_{0}}\rvert+C\left(\lVert D_h^{2}(f*\Gamma)\rVert_{L^{\infty}(B_{1})}+\lVert u\rVert_{L^{\infty}(B_{1})}\right),\\
 & \le(\overline{C}_{1}C_{4}+C)\left(\lVert D_h^{2}(f*\Gamma)\rVert_{L^{\infty}(B_{1})}+\lVert u\rVert_{L^{\infty}(B_{1})}\right).
\end{split}
\end{equation}
We consider the ``rescaled function'' defined in Lemma~\ref{lem:blowup-approximation}: 
\begin{equation}\label{eq:u_0}
u_{0}(x)\df\frac{u(x_{0}\dil{2^{-k_{0}}}x)-\ell^{x_{0}}(\delta_{2^{-k_{0}}}x)-p_{c_02^{-k_{0}}}^{x_{0}}(\dil{2^{-k_{0}}}x)}{4^{-k_{0}}}.
\end{equation}
This function coincides with $u_{2^{-k_0},x_0}$ of the same lemma. 
Now we set 
$
f_{0}(x)\df f(x_{0}\delta_{2^{-k_{0}}}x),
$
that is also defined on $B_{1}^{G}.$ We can find a harmonic function
$h_{0}$ such that
\[
v_{0}=2^{2k_{0}}v(x_{0}\delta_{2^{-k_{0}}}x)+h_{0}
\]
and $v_0$ satisfies the Dirichlet problem
\begin{equation}\label{eq:Poisson}
\begin{cases}
\Delta_{H}v_{0}=f_{0} & \text{in }B_{\sigma}^{G},\\
v_{0}=u_{0}, & \text{on }\partial B_{\sigma}^{G},
\end{cases}
\end{equation}
with $0<\sigma<1$. Notice that $v_{0}$ is also twice $L^{1}$ differentiable at $0,$
being a consequence of the twice $L^{1}$ differentiability of $v$
at $x_{0}.$ For the same reason, the twice $L^{1}$ differentiability
of $u$ at $x_{0}$ gives the twice $L^{1}$ differentiability of
$u_{0}$ at $0$. From Lemma~\ref{lem:blowup-approximation} with $\lambda=1/4$,
there exists $C_\sigma>0$ such that
\begin{equation}
\lVert D_{h}^{2}v_{0}\rVert_{L^{\infty}(B_{\sigma^2}^{G})}\le C_\sigma\left(\lVert D_{h}^{2}(f*\Gamma)\rVert_{L^{\infty}(B_{1})}+\lVert u\rVert_{L^{\infty}(B_{1})}\right).\label{eq:estimD^2v_0}
\end{equation}
For the sequel, it is now important to remark that the difference
\begin{equation}
w_{0}\df v_{0}-u_{0}\label{eq:v_0-u_0}
\end{equation}
is twice $L^{1}$ differentiable at the origin. Then we know the existence
of a polynomial 
\[
R(x)=w_{0}(0)+\sum_{j=1}^{m}X_{j}w_{0}(0)\,x_{j}+\frac{1}{2}\sum_{i,j=1}^{m}\big((X_{i}X_{j}+X_{j}X_{i})w_{0}\big)\!(0)\,x_{i}x_{j}+\sum_{l=m+1}^{m_{2}}X_{l}w_{0}(0)\,x_{l}
\]
such that we get
\begin{equation}\label{eq:limR(x)}
\frac{1}{r^{2}}\fint_{_{B_{\kappa r}}}|w_{0}(z)-R(z)|dz\to0
\end{equation}
as $r\to0^{+}$ and for an arbitrary $\kappa>0$.
The definition of $w_{0}$ immediately gives 
\[
\begin{cases}
\Delta_{\Heis}w_{0}=f_{0}\chi_{\Lambda_{2^{-k_{0}}}(x_{0})} & \text{in }B_{\sigma}^{G},\\
w_{0}=0 & \text{on }\partial B_{\sigma}^{G}.
\end{cases}
\]
{\bf Claim}: for a fixed $0<\alpha<1$, there
exist $l_{0}\ge1$ and $C>0$, depending on $\alpha$ and on universal constants,
such that for $\tau=2^{-l_{0}}$ and for every $k\in\N\sm\left\{ 0\right\}$,
there exist harmonic polynomials $q_{k}$ with the property that 
\begin{equation}\label{eq:Inductionq_k}
\lVert w_{0}-q_{k}\rVert_{L^{\infty}(B_{\tau^{k}}^{G})}\le C\left(\lVert D_{h}^{2}(f*\Gamma)\rVert_{L^{\infty}(B_{1})}+\lVert u\rVert_{L^{\infty}(B_{1})}\right)\tau^{(2+\alpha)(k-1)}
\end{equation}
where the constants are independent of $x_{0}$.

To prove \eqref{eq:Inductionq_k} by induction,
we need first to establish the case $k=1$. Here we choose the null
harmonic polynomial $q_{1}=0.$ We consider the decomposition \eqref{eq:v_0-u_0}
and observe that standard $L^{\infty}$ estimates for $v_0$ are available,
since it solves \eqref{eq:Poisson}.
Indeed, we may further decompose $v_0$ into the sum of $z_0=-f_0*\Gamma$ and
of a harmonic function $h_0$ such that $h_0|_{\partial B^G_\sigma}=u_0-z_0$.
Then we apply the sub-quadratic growth estimate \eqref{eq:quadratic-growthBG} of
Remark~\ref{rem:subquadraticnoZero} where we fix $\lambda=1/4$. 
This leads us to the following estimate
\[
\|v_0\|_{L^\infty(B_\sigma)}\le C_{1\sigma}\left(\|D_{h}^{2}(f*\Gamma)\|_{L^{\infty}(B_{1})}+\|u\|_{L^{\infty}(B_{1})}\right).
\]
Then using again estimate \eqref{eq:quadratic-growthBG}, we obtain 
\begin{align*}
\|w_{0}\|_{L^{\infty}(B_{\sigma}^{G})} & \le\|v_{0}\|_{L^{\infty}(B_{r_{1}}^{G})}+\|u_{0}\|_{L^{\infty}(B_{r_{1}}^{G})}\\
 & \le C_{2\sigma} \left(\|D_{h}^{2}(f*\Gamma)\|_{L^{\infty}(B_{1})}+\|u\|_{L^{\infty}(B_{1})}\right).
\end{align*}
Taking $\sigma=\tau=2^{-l_0}$, the estimate \eqref{eq:Inductionq_k} is established for $k=1$.
We may take $l_{0}\in\N$ possibly larger, such that 
\begin{equation}\label{eq:lambdaalphac}
\tau=2^{-l_{0}}\le \frac1{2|B_1^G|^{1/Q}}.
\end{equation}
In view of Proposition~\ref{prop:derivEst}, there exists a universal constant 
$c>0$ such that 
\begin{equation}
\lVert D_{h}^{3}H\rVert_{L^{\infty}(B_{1/2}^{G})}\le c\lVert H\rVert_{L^{\infty}(B_{1}^{G})}\label{eq:HarmonicEst}
\end{equation}
for any harmonic function $H$ on $B_{1}^{G}$.
Now we assume the statement \eqref{eq:Inductionq_k} is true for any fixed $k\ge1$
and define 
\[
w_{k}(x)\df\frac{w_{0}(\dil{\tau^{k}}x)-q_{k}(\dil{\tau^{k}}x)}{\tau^{(2+\alpha)(k-1)}}
\]
on $\overline{B_{1}^{G}}$. We choose the harmonic function $h_{k}$
such that
\[
\begin{cases}
\Delta_{\Heis}h_{k}=0 & \text{in }B_{1}^{G},\\
h_{k}=w_{k} & \text{on }\partial B_{1}^{G}.
\end{cases}
\]
From the definition of $w_k$, we get 
\[
\Delta_{H}w_{k}=\tau^{2-\alpha(k-1)}f(x_{0}\dil{2^{-k_{0}}\tau^{k}}\cdot)\chi_{\Lambda_{2^{-k_{0}}\tau^{k}}(x_{0})}
\]
on $B_{1}^{G}$, and the induction assumption yields
\begin{equation}\label{eq:inductAss1}
\lVert w_{k}\rVert_{L^{\infty}(B_{1}^{G})}\le C\left(\lVert D_h^{2}(f*\Gamma)\rVert_{L^{\infty}(B_{1})}+\lVert u\rVert_{L^{\infty}(B_{1})}\right).
\end{equation}
Clearly $w_{k}-h_{k}$ vanishes on $\der B_{1}^{G}$. Taking into account our choice
of $i_0$ such that $2^{-i_0}\le r_0$, the decay estimate of the coincidence set 
\eqref{eq:Decayr} applies in particular for $\beta=4$ and for every $r\in(0,2^{-k_0}]$, that is
\begin{equation}\label{eq:decaybeta=4}
\lvert\Lambda_{r/2}(x_{0})\rvert\le\frac{\lvert\Lambda_{r}(x_{0})\rvert}{2^{4 Q}}.
\end{equation}
Arguing as before for $v_0$, we can decompose $w_k-h_k$ into the sum of a harmonic function and a convolution with the fundamental solution, therefore standard convolution estimates yield the first of the following inequalities
\begin{align}
\lVert w_{k}-h_{k}\rVert_{L^{\infty}(B_{1}^{G})} & \le C\lVert\tau^{2-\alpha(k-1)}f(x_{0}\dil{2^{-k_{0}-l_0k}})\chi_{\Lambda_{2^{-k_{0}-l_0k}}(x_{0})}\rVert_{L^{Q}(B^G_{1})}\nonumber \\
 & \le C\tau^{-\alpha(k-1)}\lVert f\rVert_{L^{\infty}(B_{1})}\lvert\Lambda_{2^{-k_{0}-l_0k}}(x_{0})\rvert^{1/Q}\nonumber \\
 & \le C2^{\alpha l_{0}(k-1)}\lVert f\rVert_{L^{\infty}(B_{1})}2^{-4 l_{0}k} \lvert\Lambda_{2^{-k_{0}}}(x_{0})\rvert^{1/Q}\nonumber \\
 & \le C|B^G_1|^{1/Q} \lVert f\rVert_{L^{\infty}(B_{1})}\tau^{k(4-\alpha)+\alpha}\label{eq:hk-wk}\\
 & \le\frac{C}{2}\lVert f\rVert_{L^{\infty}(B_{1})}\tau^{2+\alpha},
\end{align}
where the third inequality is a consequence of \eqref{eq:decaybeta=4}
and the last inequality follows from \eqref{eq:lambdaalphac}.
Combining the estimate \eqref{eq:HarmonicEst} for harmonic functions, the maximum principle 
and our induction assumption as stated in \eqref{eq:inductAss1}, we get 
\begin{align*}
\lVert D_{h}^{3}h_{k}\rVert_{L^{\infty}(B_{1/2}^{G})} & \le c\lVert h_{k}\rVert_{L^{\infty}(B_{1}^{G})}\le c\lVert w_{k}\rVert_{L^{\infty}(B_{1}^{G})}\\
 & \le cC\left(\lVert D_h^{2}(f*\Gamma)\rVert_{L^{\infty}(B_{1})}+\lVert u\rVert_{L^{\infty}(B_{1})}\right).
\end{align*}
Define $\overline{q}_{k}(x)$ as the second order Taylor polynomial
of $h_{k}$ at the origin. In particular, $\overline{q}_k$ is harmonic.
The previous estimates joined with the application of the stratified Taylor inequality
stated in \cite[Corollary~1.44~with $k=2$ and $x=0$]{MR657581} give
\begin{align*}
\lVert h_{k}-\overline{q}_{k}\rVert_{L^{\infty}(B_{\tau}^{G})} & \le C'_2c\,C\left(\lVert D_{h}^{2}(f*\Gamma)\rVert_{L^{\infty}(B_{1})}+\lVert u\rVert_{L^{\infty}(B_{1})}\right)\tau^{3}\\
 & \le\frac{C}{2}\left(\lVert D_{h}^{2}(f*\Gamma)\rVert_{L^{\infty}(B_{1})}+\lVert u\rVert_{L^{\infty}(B_{1})}\right)\tau^{2+\alpha},
\end{align*}
where we have chosen $l_0$ possibly larger, such that the following conditions
\begin{equation}\label{eq:lambdaalphac2}
b^3\tau=b^{3}2^{-l_0}<1/2\quad\text{and}\quad \tau=2^{-l_{0}}\le\left(\frac{1}{2cC'_2}\right)^{1/(1-\alpha)}
\end{equation}
also hold. The constants $b$ and $C'_2$ are from 
\cite[Corollary~1.44 with $k=2$ and $x=0$]{MR657581} and 
this corollary is applied with the gauge distance $d_G$. 
We stress that $l_{0}$ does not depend on either $k_{0}$ or $x_{0}$. 
This is very important
for the final estimate of $D_{h}^{2}u(x_{0})$. As a consequence, we obtain 
\begin{align*}
\lVert w_{k}-\overline{q}_{k}\rVert_{L^{\infty}(B_{\tau}^{G})} & \le\lVert w_{k}-h_{k}\rVert_{L^{\infty}(B_{\tau}^{G})}+\lVert h_{k}-\overline{q}_{k}\rVert_{L^{\infty}(B_{\tau}^{G})}\\
 & \le\frac{C}{2}\lVert f\rVert_{L^{\infty}(B_{1})}\tau^{2+\alpha}+\frac{C}{2}\left(\lVert D_{h}^{2}(f*\Gamma)\rVert_{L^{\infty}(B_{1})}+\lVert u\rVert_{L^{\infty}(B_{1})}\right)\tau^{2+\alpha}\\
 & \le C\left(\lVert D_{h}^{2}(f*\Gamma)\rVert_{L^{\infty}(B_{1})}+\lVert u\rVert_{L^{\infty}(B_{1})}\right)\tau^{2+\alpha}.
\end{align*}
Taking into account the definition of $w_{k},$ we have proved that
\[
\left\Vert \frac{w_{0}(\dil{\tau^{k}}\cdot)-q_{k}(\dil{\tau^{k}}\cdot)}{\tau^{(2+\alpha)(k-1)}}-\overline{q}_{k}\right\Vert _{L^{\infty}(B_{\tau}^{G})}\le C\left(\lVert D_{h}^{2}(f*\Gamma)\rVert_{L^{\infty}(B_{1})}+\lVert u\rVert_{L^{\infty}(B_{1})}\right)\tau^{2+\alpha},
\]
from which we infer that
\[
\lVert w_{0}-q_{k}-\tau^{(2+\alpha)(k-1)}\overline{q}_{k}(\dil{\tau^{-k}}\cdot)\rVert_{L^{\infty}(B_{\tau^{k+1}}^{G})}\le C\left(\lVert D_{h}^{2}(f*\Gamma)\rVert_{L^{\infty}(B_{1})}+\lVert u\rVert_{L^{\infty}(B_{1})}\right)\tau^{(2+\alpha)k}.
\]
If we define the new polynomial 
\[
q_{k+1}(x)\df q_{k}(x)+\tau^{(2+\alpha)(k-1)}\overline{q}_{k}(\dil{\tau^{-k}}x),
\]
then the induction step is proved and this concludes the proof of our claim.
By the same previous argument, we have another universal constant $c'>0$ such that
\begin{align}
 &\phantom{{}={}} \max\left\{ \lVert h_{k}\rVert_{L^{\infty}(B_{1/2}^{G})},\lVert\nabla_h h_{k}\rVert_{L^{\infty}(B_{1/2}^{G})},\lVert D_{h}^{2}h_{k}\rVert_{L^{\infty}(B_{1/2}^{G})}\right\} \nonumber \\
 & \le c'\lVert w_{k}\rVert_{L^{\infty}(B_{1}^{G})}\label{eq:DiffEstim h_k}\\
 & \le c'C\left(\lVert D_h^{2}(f*\Gamma)\rVert_{L^{\infty}(B_{1})}+\lVert u\rVert_{L^{\infty}(B_{1})}\right).\nonumber 
\end{align}
We introduce the following notation:
\begin{align*}
q_{k}(x) & =a^{k}+\sum_{i=1}^{m}b_{i}^{k}x_{i}+\frac{1}{2}\sum_{i,j=1}^{m}c_{ij}^{k}x_{i}x_{j}+\sum_{l=m+1}^{m_{2}}c_{l}^{k}x_{l},\\
\overline{q}_{k}(x) & =\overline{a}^{k}+\sum_{i=1}^{m}\overline{b}_{i}^{k}x_{i}+\frac{1}{2}\sum_{i,j=1}^{m}\overline{c}_{ij}^{k}x_{i}x_{j}+\sum_{l=m+1}^{m_{2}}\overline{c}_{l}^{k}x_{l}.
\end{align*}
From the definition of $\bar{q}_{k}$ and taking into account the estimates \eqref{eq:DiffEstim h_k}, we get
\begin{equation}
\max_{i,j,l}\{\overline{a}^{k},\overline{b}_{i}^{k},\overline{c}_{ij}^{k},\overline{c}_{l}^{k}\}\le c'C\left(\lVert D_h^{2}(f*\Gamma)\rVert_{L^{\infty}(B_{1})}+\lVert u\rVert_{L^{\infty}(B_{1})}\right).\label{eq:abar_kij}
\end{equation}
Consequently, differentiating the equality
\[
q_{k+1}(x)-q_{k}(x)=\tau^{(2+\alpha)(k-1)}\overline{q}_{k}(\dil{\tau^{-k}}x),
\]
with the differential operators $X_{i}$, $X_{i}X_{j}$ for $i,j=1,\ldots,m$
and $X_{l}$ for $l=m+1,\ldots,m_{2}$, and evaluating all the equalities at the origin,
we get
\begin{align*}
\lvert a^{k+1}-a^{k}\rvert & \le c'C\left(\lVert D_h^{2}(f*\Gamma)\rVert_{L^{\infty}(B_{1})}+\lVert u\rVert_{L^{\infty}(B_{1})}\right)\tau^{(2+\alpha)(k-1)},\\
\max_{1\le i\le m}\lvert b_{i}^{k+1}-b_{i}^{k}\rvert & \le c'C\tau^{-1}\left(\lVert D_h^{2}(f*\Gamma)\rVert_{L^{\infty}(B_{1})}+\lVert u\rVert_{L^{\infty}(B_{1})}\right)\tau^{(1+\alpha)(k-1)},\\
\max_{1\le i,j\le m}\lvert c_{ij}^{k+1}-c_{ij}^{k}\rvert & \le c'C\tau^{-2}\left(\lVert D_h^{2}(f*\Gamma)\rVert_{L^{\infty}(B_{1})}+\lVert u\rVert_{L^{\infty}(B_{1})}\right)\tau^{\alpha(k-1)},\\
\max_{m+1\le l\le m_{2}}\lvert c_{l}^{k+1}-c_{l}^{k}\rvert & \le c'C\tau^{-2}\left(\lVert D_h^{2}(f*\Gamma)\rVert_{L^{\infty}(B_{1})}+\lVert u\rVert_{L^{\infty}(B_{1})}\right)\tau^{\alpha(k-1)}.
\end{align*}
Representing any of these coefficients as $\gamma^{k}$, we notice
that they are Cauchy sequences converging to some $\gamma$. In addition,
for any of them we have 
\begin{equation}
|\gamma^{k}-\gamma|\le\frac{c'C\tau^{-2}}{1-\tau^{\alpha}}\left(\lVert D_h^{2}(f*\Gamma)\rVert_{L^{\infty}(B_{1})}+\lVert u\rVert_{L^{\infty}(B_{1})}\right)\tau^{(\alpha+l)(k-1)}.\label{eq:gammakl}
\end{equation}
In these estimates, we have set $l=0$ when $\gamma^{k}=$ $c_{l}^{k},c_{ij}^{k},$
$l=1$ for \textbf{$\gamma^{k}=b_{i}^{k}$ }and $l=2$ for $\gamma^{k}=a_{i}^{k}.$
As a consequence, the polynomials $q_{k}$ uniformly converge on compact
sets to a polynomial $\tilde{q}$, that has the form 
\[
\tilde{q}(x)=\tilde{a}+\sum_{i=1}^{m}\tilde{b}_{i}x_{i}+\frac{1}{2}\sum_{i,j=1}^{m}\tilde{c}_{ij}x_{i}x_{j}+\sum_{l=m+1}^{m_{2}}\tilde{c}_{l}x_{l}.
\]
Any coefficient $\gamma$ of $\tilde{q}$, can be written for instance
as $\gamma^{2}+\gamma-\gamma^{2}.$ By \eqref{eq:gammakl}, we can find a universal estimate for $\gamma-\gamma^{2}$, that depends on $\alpha$. We also observe that the
coefficients of $q_{2}$ are given by the formula $q_{2}=\overline{q}_{1}\circ\delta_{\tau^{-1}}$.
The estimate \eqref{eq:abar_kij} for $k=2$ and the fact that $\tau=2^{-l_{0}}$
can be universally fixed, independently of $x_{0}$, finally lead us to
the estimate
\begin{equation}
|D_{h}^{2}\tilde{q}|\le C_{\tau,\alpha}\left(\lVert D_h^{2}(f*\Gamma)\rVert_{L^{\infty}(B_{1})}+\lVert u\rVert_{L^{\infty}(B_{1})}\right)\label{eq:D^2_htildeq}
\end{equation}
for a suitable geometric constant $C_{\tau,\alpha}>0$ depending on the constants of \eqref{eq:gammakl} and \eqref{eq:abar_kij}.
From \eqref{eq:gammakl}, defining 
\[
\overline{C}_{\tau,\alpha,f,u}=\frac{c'C\tau^{-2}\left(\lVert D_h^{2}(f*\Gamma)\rVert_{L^{\infty}(B_{1})}+\lVert u\rVert_{L^{\infty}(B_{1})}\right)}{1-\tau^{\alpha}},
\]
we can establish the following quantitative estimate on small balls: 
\[
\begin{split}\|q_{k}-\tilde{q}\|_{L^{\infty}(B^G_{\tau^{k+1}})}\le & \,\overline{C}_{\tau,\alpha,f,u}\Bigg(\tau^{(\alpha+2)(k-1)}+\tau^{(\alpha+1)(k-1)}\sum_{i=1}^{m}\|x_{i}\|_{L^{\infty}(B^G_{\tau^{k+1}})}\\
+ & \tau^{\alpha(k-1)}\sum_{i,j=1}^{m}\|x_{i}x_{j}\|_{L^{\infty}(B^G_{\tau^{k+1}})}+\tau^{\alpha(k-1)}\sum_{l=m+1}^{m_{2}}\|x_{l}\|_{L^{\infty}(B^G_{\tau^{k+1}})}\Bigg)\\
\le & \widetilde{C}\,\overline{C}_{\tau,\alpha,f,u}\,\tau^{(2+\alpha)(k-1)},
\end{split}
\]
where we have used the proper intrinsic homogeneity of all
the monomials $x_i$, $x_ix_j$ and $x_l$, with $i,j=1,\ldots,m$ and $l=m+1,\ldots,m_2$. We would like to prove that 
\begin{equation}\label{eq:r2BG}
\lim_{r\to0}\frac{1}{r^{2}}\fint_{B^G_{\tau r}}|w_{0}(z)-\tilde{q}(z)|dz=0.
\end{equation}
Let us consider $r=\tau^{k}$ and then
\[
\begin{split}\frac{1}{r^{2}}\fint_{B^G_{\tau r}}|w_{0}(z)-\tilde{q}(z)|dz\le & \frac{1}{r^{2}}\fint_{B^G_{\tau r}}|w_{0}(z)-q_{k}(z)|dz+\frac{1}{r^{2}}\fint_{B^G_{\tau r}}|\tilde{q}(z)-q_{k}(z)|dz\\
= & \tau^{-2k}\fint_{B^G_{\tau^{k+1}}}|w_{0}(z)-q_{k}(z)|dz+\tau^{-2k}\fint_{B^G_{\tau^{k+1}}}|\tilde{q}(z)-q_{k}(z)|dz\\
\le & C\left(\lVert D_{h}^{2}(f*\Gamma)\rVert_{L^{\infty}(B_{1})}+\lVert u\rVert_{L^{\infty}(B_{1})}\right)\tau^{\alpha k}+\tau^{-2k}\|\tilde{q}-q_{k}\|_{L^{\infty}(B^G_{\tau^{k}})}
\end{split}
\]
goes to zero as $k\to\infty.$ If we choose $\kappa_0>0$ sufficiently small, then we have proved that
\[
\frac{1}{\tau^{2k}}\fint_{_{B_{\kappa_0 \tau^{2k}}}}|w_{0}(z)-R(z)|dz\le
\frac{1}{\tau^{2k}}\fint_{_{B^G_{\tau^k}}}|w_{0}(z)-R(z)|dz\to0.
\]
By the uniqueness of the second order
polynomial $R$ satisfying \eqref{eq:limR(x)} with $\kappa=\kappa_0$,
we get $\tilde{q}=R$. Taking into account \eqref{eq:estimD^2v_0} with $\sigma=\tau$
and \eqref{eq:D^2_htildeq} with $\alpha=1/2$, we obtain a new constant $C>0,$
such that
\[
\lvert D_{h}^{2}u_{0}(0)\rvert\le\lvert D_{h}^{2}v_{0}(0)\rvert+\lvert D_{h}^{2}w_{0}(0)\rvert\le C(\lVert D_{h}^{2}(f*\Gamma)\rVert_{L^{\infty}(B_{1})}+\lVert u\rVert_{L^{\infty}(B_{1})}).
\]
As a consequence, since $k_{0}>i_0$ and $i_0$ satisfies \eqref{eq:i_0}
we can apply Lemma~\ref{lem:dyadic-scale} with $r_1=2^{-k_0}$ and $r_2=c_02^{-k_0}$, so that taking into account the estimate \eqref{eq:PC_0} and the definition \eqref{eq:u_0} of $u_0$, we finally obtain a possibly larger constant, that we still denote by $C>0$, such that
\[
\lvert D_{h}^{2}u(x_{0})\rvert\le\lvert D_{h}^{2}u_{0}(0)\rvert+\lvert P_{c_02^{-k_{0}}}^{x_{0}}\rvert\le C(\lVert D_{h}^{2}(f*\Gamma)\rVert_{L^{\infty}(B_{1})}+\lVert u\rVert_{L^{\infty}(B_{1})}),
\]
concluding the proof.
\end{proof}
\bibliographystyle{amsalpha2}
\bibliography{References}

\providecommand{\bysame}{\leavevmode\hbox to3em{\hrulefill}\thinspace}
\providecommand{\MR}{\relax\ifhmode\unskip\space\fi MR }
\providecommand{\MRhref}[2]{%
  \href{http://www.ams.org/mathscinet-getitem?mr=#1}{#2}
}
\providecommand{\href}[2]{#2}
\begin{thebibliography}{DFPP08}

\bibitem[ALS13]{MR2999297}
John Andersson, Erik Lindgren, and Henrik Shahgholian, \emph{Optimal regularity
  for the no-sign obstacle problem}, Comm. Pure Appl. Math. \textbf{66} (2013),
  no.~2, 245--262. \MR{2999297}

\bibitem[BLU07]{MR2363343}
A.~Bonfiglioli, E.~Lanconelli, and F.~Uguzzoni, \emph{Stratified {L}ie groups
  and potential theory for their sub-{L}aplacians}, Springer Monographs in
  Mathematics, Springer, Berlin, 2007. \MR{2363343}

\bibitem[BB00]{MR1608289}
Marco Bramanti and Luca Brandolini, \emph{{$L^p$} estimates for nonvariational
  hypoelliptic operators with {VMO} coefficients}, Trans. Amer. Math. Soc.
  \textbf{352} (2000), no.~2, 781--822. \MR{1608289}

\bibitem[BB05]{MR2174915}
\bysame, \emph{Estimates of {BMO} type for singular integrals on spaces of
  homogeneous type and applications to hypoelliptic {PDE}s}, Rev. Mat.
  Iberoamericana \textbf{21} (2005), no.~2, 511--556. \MR{2174915}

\bibitem[BF13]{MR3100057}
Marco Bramanti and Maria~Stella Fanciullo, \emph{{\it {BMO}} estimates for
  nonvariational operators with discontinuous coefficients structured on
  {H}\"ormander's vector fields on {C}arnot groups}, Adv. Differential
  Equations \textbf{18} (2013), no.~9-10, 955--1004. \MR{3100057}

\bibitem[Caf89]{MR1005611}
Luis~A. Caffarelli, \emph{Interior a priori estimates for solutions of fully
  nonlinear equations}, Ann. of Math. (2) \textbf{130} (1989), no.~1, 189--213.
  \MR{1005611 (90i:35046)}

\bibitem[CKS00]{MR1745013}
Luis~A. Caffarelli, Lavi Karp, and Henrik Shahgholian, \emph{Regularity of a
  free boundary with application to the {P}ompeiu problem}, Ann. of Math. (2)
  \textbf{151} (2000), no.~1, 269--292. \MR{1745013}

\bibitem[DGP07]{MR2323535}
Donatella Danielli, Nicola Garofalo, and Arshak Petrosyan, \emph{The
  sub-elliptic obstacle problem: {$C^{1,\alpha}$} regularity of the free
  boundary in {C}arnot groups of step two}, Adv. Math. \textbf{211} (2007),
  no.~2, 485--516. \MR{2323535}

\bibitem[DGS03]{MR1976081}
Donatella Danielli, Nicola Garofalo, and Sandro Salsa, \emph{Variational
  inequalities with lack of ellipticity. {I}. {O}ptimal interior regularity and
  non-degeneracy of the free boundary}, Indiana Univ. Math. J. \textbf{52}
  (2003), no.~2, 361--398. \MR{1976081}

\bibitem[DFPP08]{MR2434049}
Marco Di~Francesco, Andrea Pascucci, and Sergio Polidoro, \emph{The obstacle
  problem for a class of hypoelliptic ultraparabolic equations}, Proc. R. Soc.
  Lond. Ser. A Math. Phys. Eng. Sci. \textbf{464} (2008), no.~2089, 155--176.
  \MR{2434049}

\bibitem[FS12]{figsha}
Alessio Figalli and Henrik Shahgholian, \emph{A {G}eneral {C}lass of {F}ree
  {B}oundary {P}roblems for {F}ully {N}onlinear {E}lliptic {E}quations}.

\bibitem[Fol75]{MR0494315}
G.~B. Folland, \emph{Subelliptic estimates and function spaces on nilpotent
  {L}ie groups}, Ark. Mat. \textbf{13} (1975), no.~2, 161--207. \MR{0494315}

\bibitem[FS82]{MR657581}
G.~B. Folland and Elias~M. Stein, \emph{Hardy spaces on homogeneous groups},
  Mathematical Notes, vol.~28, Princeton University Press, Princeton, N.J.;
  University of Tokyo Press, Tokyo, 1982. \MR{657581}

\bibitem[Fre72]{MR0318650}
Jens Frehse, \emph{On the regularity of the solution of a second order
  variational inequality}, Boll. Un. Mat. Ital. (4) \textbf{6} (1972),
  312--315. \MR{0318650 (47 \#7197)}

\bibitem[Fre13]{MR3093377}
Marie Frentz, \emph{Regularity in the obstacle problem for parabolic
  non-divergence operators of {H}\"ormander type}, J. Differential Equations
  \textbf{255} (2013), no.~10, 3638--3677. \MR{3093377}

\bibitem[FGN12]{MR2891355}
Marie Frentz, Elin G\"otmark, and Kaj Nystr\"om, \emph{The obstacle problem for
  parabolic non-divergence form operators of {H}\"ormander type}, J.
  Differential Equations \textbf{252} (2012), no.~9, 5002--5041. \MR{2891355}

\bibitem[FNPP10]{MR2658144}
Marie Frentz, Kaj Nystr\"om, Andrea Pascucci, and Sergio Polidoro,
  \emph{Optimal regularity in the obstacle problem for {K}olmogorov operators
  related to {A}merican {A}sian options}, Math. Ann. \textbf{347} (2010),
  no.~4, 805--838. \MR{2658144}

\bibitem[Fri88]{MR1009785}
Avner Friedman, \emph{Variational principles and free-boundary problems},
  second ed., Robert E. Krieger Publishing Co., Inc., Malabar, FL, 1988.
  \MR{1009785}

\bibitem[GN96]{MR1404326}
Nicola Garofalo and Duy-Minh Nhieu, \emph{Isoperimetric and {S}obolev
  inequalities for {C}arnot-{C}arath\'{e}odory spaces and the existence of
  minimal surfaces}, Comm. Pure Appl. Math. \textbf{49} (1996), no.~10,
  1081--1144. \MR{1404326}

\bibitem[GT01]{gilbarg:01}
David Gilbarg and Neil~S. Trudinger, \emph{Elliptic partial differential
  equations of second order}, Springer, 2001.

\bibitem[GS05]{MR2129734}
Bj\"orn Gustafsson and Harold~S. Shapiro, \emph{What is a quadrature domain?},
  Quadrature domains and their applications, Oper. Theory Adv. Appl., vol. 156,
  Birkh\"auser, Basel, 2005, pp.~1--25. \MR{2129734}

\bibitem[Jer86]{MR850547}
David Jerison, \emph{The {P}oincar\'e inequality for vector fields satisfying
  {H}\"ormander's condition}, Duke Math. J. \textbf{53} (1986), no.~2,
  503--523. \MR{850547}

\bibitem[KS00]{MR1786735}
David Kinderlehrer and Guido Stampacchia, \emph{An introduction to variational
  inequalities and their applications}, Classics in Applied Mathematics,
  vol.~31, Society for Industrial and Applied Mathematics (SIAM), Philadelphia,
  PA, 2000, Reprint of the 1980 original. \MR{1786735}

\bibitem[LM00]{MR1785405}
Ermanno Lanconelli and Daniele Morbidelli, \emph{On the {P}oincar\'e inequality
  for vector fields}, Ark. Mat. \textbf{38} (2000), no.~2, 327--342.
  \MR{1785405}

\bibitem[Lu96]{MR1425620}
Guozhen Lu, \emph{Embedding theorems into {L}ipschitz and {BMO} spaces and
  applications to quasilinear subelliptic differential equations}, Publ. Mat.
  \textbf{40} (1996), no.~2, 301--329. \MR{1425620}

\bibitem[Mag05]{MR2146126}
Valentino Magnani, \emph{Differentiability from the representation formula and
  the {S}obolev-{P}oincar\'{e} inequality}, Studia Math. \textbf{168} (2005),
  no.~3, 251--272. \MR{2146126}

\bibitem[PS07]{MR2369892}
Arshak Petrosyan and Henrik Shahgholian, \emph{Geometric and energetic criteria
  for the free boundary regularity in an obstacle-type problem}, Amer. J. Math.
  \textbf{129} (2007), no.~6, 1659--1688. \MR{2369892}

\bibitem[PSU12]{MR2962060}
Arshak Petrosyan, Henrik Shahgholian, and Nina Uraltseva, \emph{Regularity of
  free boundaries in obstacle-type problems}, Graduate Studies in Mathematics,
  vol. 136, American Mathematical Society, Providence, RI, 2012. \MR{2962060}

\bibitem[Rod87]{MR880369}
Jos\'{e}-Francisco Rodrigues, \emph{Obstacle problems in mathematical physics},
  North-Holland Mathematics Studies, vol. 134, North-Holland Publishing Co.,
  Amsterdam, 1987, Notas de Matem\'{a}tica [Mathematical Notes], 114.
  \MR{880369}

\bibitem[Sak91]{MR1097025}
Makoto Sakai, \emph{Regularity of a boundary having a {S}chwarz function}, Acta
  Math. \textbf{166} (1991), no.~3-4, 263--297. \MR{1097025}

\bibitem[Var84]{MR746308}
V.~S. Varadarajan, \emph{Lie groups, {L}ie algebras, and their
  representations}, Graduate Texts in Mathematics, vol. 102, Springer-Verlag,
  New York, 1984, Reprint of the 1974 edition. \MR{746308}

\end{thebibliography}

\end{document}